\documentclass[11pt]{amsart}
\usepackage{amssymb}
\usepackage{amscd}
\usepackage{verbatim}
\usepackage[usenames,dvipsnames]{pstricks}
\usepackage{epsfig}
\usepackage{pst-grad} % For gradients
 \usepackage{pst-plot} % For axes
\usepackage{verbatim}
\usepackage{curves}
\usepackage{geometry}
\usepackage{tikz}
\usetikzlibrary{patterns}

\usepackage{enumerate}

\usepackage{geometry}
\newtheorem{thm}{Theorem}[section]
\newtheorem{lemma}[thm]{Lemma}
\newtheorem{prop}[thm]{Proposition}
\newtheorem{cor}[thm]{Corollary}

\theoremstyle{definition}

\theoremstyle{remark}

\numberwithin{equation}{section}

\DeclareMathOperator{\im}{Im}

\DeclareMathOperator{\loc}{loc}

\def \intx {\stackrel{\circ}{X}}

\def \im {\operatorname{Im}}

\def \restr {{\left.\right|}}
\def \ogamma {\overline{\Gamma}}
\def \mrn {{\mathbb R}^n}

\def \mr {{\mathbb R}}
\def \loc {{\operatorname{loc}}}

\def \mcs {{\mathcal S}}
\def \mcsg {{\mathcal S}\left.\right|_{\mr \times \Gamma}}
\def \mcp {{\mathcal P}}
\def \mcm {{\mathcal M}}
\def \mca {{\mathcal A}}

\def \mcd {{\mathcal D}}
\def \mcf {{\mathcal F}}

\def \mcl {{\mathcal L}}
\def \mcr {{\mathcal R}}
\def \mco {{\mathcal O}}

\def \mcq {{\mathcal Q}}

\def \mcc {{\mathcal C}}
\def \mh {{\mathbb H}}

\def \mcn {{\mathbb C}^n}

\def \eps {\varepsilon}   
\def \la {\lambda}

\def \lan {\langle}   
\def \ran {\rangle}   
\def \del {\delta}   

\def \det {\operatorname{det}}

\newcommand{\id}{\operatorname{Id}}
\newcommand{\dvol}{d\operatorname{vol}}

\def \ha{ {\frac{1}{2}}}
\def \tha{ {\frac{3}{2}}}
\def \oq {\frac{1}{4}}

\def \p {\partial}

\def \rao#1 {\frac{\p}{\p #1} #1}

\numberwithin{equation}{section}

\usepackage{latexsym,eucal,amsfonts,amssymb,amsmath,graphicx}
\newcommand{\vphi}{\varphi}

\newcommand{\supp}{\textrm{Supp}}

\newcommand{\beq}{\begin{equation}}
  \newcommand{\eeq}{\end{equation}}

\def \wtp {\widetilde{P}}
\def \intx {\stackrel{\circ}{X}}

\def \restr {{\left.\right|}}
\def \mrn {{\mathbb R}^n}

\def \mr {{\mathbb R}}

\def \wtpsi {{\widetilde{\Psi}}}
\def \wtw {{\widetilde{W}}}
\def \wtq {{\widetilde{Q}}}

\def \mcs {{\mathcal S}}
\def \mcp {{\mathcal P}}
\def \mcm {{\mathcal M}}
\def \mcn {{\mathcal N}}
\def \mca {{\mathcal A}}

\def \mcd {{\mathcal D}}
\def \mcf {{\mathcal F}}

\def \mcl {{\mathcal L}}
\def \mcr {{\mathcal R}}
\def \mco {{\mathcal O}}

\def \mcq {{\mathcal Q}}

\def \mh {{\mathbb H}}
\def \intx {\overset{\circ}{X}}
\def \lcg {L^2(\mr \times \Gamma)}
\def \mb {{\mathbb B}}

\def \eps {\varepsilon}
\def \la {\lambda}

\def \novt {\frac{n}{2}}
\def \nsq {\frac{n^2}{4}}
\def \lan {\langle}
\def \ran {\rangle}
\def \del {\delta}

\def \det {\operatorname{det}}

\def \tx {\tilde{x}}
\def \wth {\widetilde{H}}

\input epsf
\usepackage{graphicx}
\usepackage[latin1]{inputenc}
\title[Inverse Scattering  with partial data]{Inverse Scattering with Partial data on  Asymptotically Hyperbolic Manifolds}
\author{Raphael Hora}
\email{rhora@math.purdue.edu}
\author{Ant\^onio S\'a Barreto}
\email{sabarre@math.purdue.edu}
\address{Department of Mathematics, Purdue University \newline
\indent 150 North University Street, West Lafayette IN  47907, USA}
\date{}
\begin{document}
\begin{abstract} We prove that the scattering matrix at all energies restricted to an open subset of the boundary determines an asymptotically hyperbolic manifold modulo isometries that are equal to the identity on the open subset where the scattering matrix is known.
\end{abstract}
\maketitle
\section{Introduction}\label{intro}

As the name suggests, asymptotically hyperbolic manifolds are modeled by the hyperbolic space at infinity.    The  ball model of the hyperbolic space is given by
\begin{gather*}
\intx= \mb^{n+1}=\{z \in \mr^{n+1}: |z|<1\} \text{ equipped with the metric } g= \frac{4 dz^2}{(1-|z|^2)^2}.
\end{gather*}
We replace $\overline{\mb^{n+1}}$ with a $C^\infty,$ connected, compact  manifold  $X$ with boundary $\p X,$ of dimension $n+1.$  We pick  a  function $x\in C^\infty (X)$  such that  $x>0$ in the interior of $X,$  $\{x=0\}= \p X$, and $dx\neq 0$ at $\partial X.$ Such a function will be called a boundary defining function.   In the model above $x=1-|z|^2.$  If   $g$ is a Riemannian metric on the interior of $X$ such that
\begin{gather}
x^2 g=H \label{defmetric}
\end{gather}
where $H$ is non-degenerate up to $\p X,$ then according to \cite{mazzmel}  $g$ is complete and its sectional curvatures approach $|dx|_{H},$ as $x\downarrow 0.$   In particular, when 
\begin{gather}
\displaystyle |dx|_{H}=1 \text{ at } \partial X, \label{asympt-curv}
\end{gather}  
the sectional curvatures converge to $-1$ at the boundary.    Such a Riemannian manifold $(X,g)$ for which \eqref{defmetric} and \eqref{asympt-curv} hold  is called asymptotically hyperbolic.  We are interested in studying the long time behavior of solutions of the wave equation on asymptotically hyperbolic manifolds, and the behavior of the metric, and hence of its geodesics, near $\p X$  influences how waves scatter. 

Notice that $g$ does not determine $H$ and it  is interesting to inquire about the behavior of $H=x^2g$ at $\p X.$ Since
  any two defining functions of $\p X,$ $x$ and $\tilde{x},$  satisfy $x=e^\omega\tilde{x},$ with $\omega\in C^\infty(X),$  the corresponding $H$ and $\widetilde{H}$  must satisfy ${H}|_{\p X}= e^{2\omega(0,y)} \wth|_{\p X}.$ Hence $H|_{\p X}$ is determined up to a conformal factor which depends on the choice of $x.$    We recall the construction of boundary normal coordinates for $g$ in this setting given by Graham \cite{Gra}. We have $H=x^2g=e^{2\omega}{\tilde{x}}^2 g,$ and hence $H=e^{2\omega} \widetilde{H}.$ Since $dx= e^\omega( \tilde{x} d\omega+ d \tilde{x}),$ we have
\begin{gather*}
|dx|^2_H=|d \tilde{x}+ \tilde{x} d\omega|_{\wth}^2= |d \tilde{x}|^2_{\wth} + {\tilde{x}}^2|d\omega|_{\wth}^2+ 2 \tilde{x} (\nabla_{\wth} \tilde{x}) \omega.
\end{gather*}
Hence,
\begin{gather*}
|dx|_{H}=1 \;  \text{ if and only if  } \;  2 (\nabla_{\wth} \tx) \omega+ \tx |d\omega|_{\wth}^2= \frac{1-|d\tx|^2_{\wth}}{\tx}, \;\ \omega\restr_{\p X}=0.
\end{gather*}
Since by assumption $|d\tilde{x}|_{\wth}=1$ at $\p X,$ this is a non-characteristic ODE, and hence it has a solution in a neighborhood of $\p X.$     So we conclude that  fixed  a representative $h_0$ of the conformal class of $H|_{\p X},$ there exists $\epsilon>0,$ a neighborhood $U_\eps$ of $\p X$ and a map $\Psi:  [0,\epsilon)\times \partial X \longrightarrow U_\eps$ such that
\begin{equation}\label{metric}
\Psi^* g=\frac{dx^2}{x^2}+\frac{h(x,y,dy)}{x^2},\quad h_{0}=h(0,y,dy),
\end{equation}
where $h$ is a one-parameter family of metrics on the boundary $\partial X.$     This construction works equally well in a neighborhood of an open subset $\Gamma \subset \p X.$

 The spectrum of the Laplacian for this type of manifolds
  was studied by Mazzeo and Mazzeo and Melrose in \cite{mazzeo1,mazzeo2,mazzmel} and more recently by Bouclet \cite{bouclet}.  The spectrum of $\Delta_g,$ denoted by $\sigma(\Delta_g)$  consists of a finite point spectrum $\sigma_{pp}(\Delta_g)$ and an absolutely continuous spectrum $\sigma_{ac}(\Delta_g)$ satisfying
\begin{gather}
\sigma_{ac}(\Delta_g)=[\frac{n^2}{4},\infty), \quad \sigma_{pp}(\Delta_g)\subset (0, \frac{n^2}{4}). \label{l2decomp}
\end{gather}

It follows from \eqref{l2decomp} and the spectral theorem that if $\im \la <<0,$ the resolvent for $\Delta_g,$ denoted by
\begin{gather}
R(\la)=\left(\Delta_g-\la^2-\frac{n^2}{4}\right)^{-1}, \label{resolvent}
\end{gather}
is a bounded operator in $L^2(X).$  The continuation of $R(\la),$ as an operator 
$$R(\la): C_0^\infty(\intx) \longrightarrow C^{-\infty}(\intx),$$
from $\im \la<<0$ to the complex plane was first studied by Mazzeo and Melrose in \cite{mazzmel}, later by Guillarmou \cite{colin2} and more recently by Vasy \cite{Vasymeromorphic}

 The structure of the generalized eigenfunctions and the analogue of the Sommerfeld radiation condition  was studied in \cite{megs,JS2}, where it was proved that  for all $f\in C^{\infty}(\partial X)$ and $\lambda\in\mathbb{R}\setminus\{0\}$, there exists a unique $v(z,\la) \in C^{\infty}(\overset{\circ}{X}),$ $z=(x,y)$ near $\p X,$ satisfying 
\begin{equation}
\begin{gathered}
\left(\Delta_{g}-\lambda^2-\frac{n^2}{4}\right) v=0\quad \text{in } \overset{\circ}{X},\\
v=x^{i\lambda+n/2}F_{+}+x^{-i\lambda+n/2}F_{-},\quad F_{\pm}\in C^{\infty}({X}),\quad F_{+}|_{\partial X}=f.
\end{gathered}\label{lapeq}
\end{equation}

The scattering matrix at energy $\lambda\neq0$ is defined as the operator
\begin{equation}
\begin{gathered}
\mca(\lambda):C^{\infty}(\partial X)\rightarrow C^{\infty}(\partial X),\\
f\longmapsto F_{-}|_{\partial X}.
\end{gathered}\label{scatmat1}
\end{equation}
This definition depends on the choice of the function $x.$   One can define $\mca(\la)$ invariantly by making it act on a certain density bundle, but we prefer to fix one function $x$ for which \eqref{metric} holds near $\p X$ and work with $\mca(\la)$ as in \eqref{scatmat1}. Notice that in view of the construction of $x,$ we are implicitly fixing a conformal representative of $\p X.$

The scattering matrix is the analogue of the Dirichlet-to-Neumann map when  $(X,g)$ is a $C^\infty$ compact Riemannian manifold with boundary $\p X.$ In this case the Laplacian $\Delta_g$ is an
elliptic differential operator and it is well known that given $f \in C^\infty(\p X),$ there exists a unique
$u \in C^\infty(X)$ such that $\Delta_g u = 0$ in $X$ and $u = f$ at $\p X.$
If $(x, y)$ are geodesic normal coordinates to $\p X,$ one can show that the entire Taylor series of
$u(x, y)$ at $\p X = \{x = 0\}$ is determined by the first two terms, $u(0, y) = f(y)$ and $\p_x u(0, y),$ and
by the equation. However, since the solution is unique, $\p_xu(0, y)$ is globally determined by the
equation and the Dirichlet data $f.$ The map $\Lambda_g : C^\infty(\p X) \longrightarrow C^\infty( \p X)$ defined by $\Lambda_g f = \p_x u\restr_{\p X}$
is called the Dirichlet-to-Neumann map.

The scattering matrix in this class of manifolds was studied in \cite{GuZwo1,megs,JS2,GraZwo,colin2}.    According to \cite{megs,JS2}, fixed $\la\in \mr,\setminus 0$ and $x,$ the operator 
$\mca(\la)$ is pseudodifferential of complex order $2i\la$ with principal symbol
\begin{gather*}
\sigma_0(\mca(\la))(y,\eta)=2^{-i\la} \frac{\Gamma(-i\la)}{\Gamma(i\la)} |\eta|_{h_0}^{2i\la},
\end{gather*}
where $\Gamma$ is the gamma function and $|\eta|_{h_0}$ is the length of the co-vector $\eta$ induced by $h_0.$   

The inverse problem of  obtaining information about a compact manifold $(X, g)$ from  the Dirichlet-to-Neumann map $\Lambda_g$ has been extensively studied, see for example \cite{uhlmann} for a survey about this question.   Joshi and S\'a Barreo \cite{JS} first studied the inverse problem  of determining an asymptotically hyperbolic manifold $(X,g)$ from the scattering matrix $\mca(\la)$ at a fixed energy $\la \in \mr\setminus 0,$ and they showed that $\mca(\la)$ determines the Taylor series of $h(x,y,dy)$ at $x=0.$  More precisely,   they proved:
\begin{thm}\label{JoSa2}(\cite{JS}) Let $(X,g)$ be an asymptotically hyperbolic manifold, let $x$ be a boundary defining function such that \eqref{metric} holds, and let $\mca(\la)$ be the scattering matrix defined in \eqref{scatmat1} for $\la\not=0.$  Let $p \in \p X$ and let $U\subset \p X$ be an open subset with $p \in U,$ and $a(\la, y,\eta)$ be the full symbol of $\mca(\la)$ with $(y,\eta)\in T^* U\setminus 0.$ Then there exists a discrete set $\mcq\subset \mr  $ such that if $\la \in \mr\setminus \mcq,$ the Taylor series of the tensor $h(x,y,dy)$ at $x=0,$ for $y \in U$ is determined by $a(\la,y,\eta).$
\end{thm}

Much more can be said in the case when $\mca(\la)$ is known for every $\la\in \mr\setminus 0.$   The main result of \cite{sa} is
\begin{thm}\label{inverse-full}(\cite{sa})
Let $(X_{1},g_1)$ and $(X_2,g_2)$ be asymptotically hyperbolic manifolds and suppose that $\p X_1=\p X_2=M$ (as manifolds). Let $x_j\in C^{\infty}(X_j)$, $j=1,2$, be a defining function of $\partial X_{j}$ for which \eqref{metric} holds, and let $A_{j}(\lambda)$, $j=1,2$, $\lambda\in\mathbb{R}\setminus\{0\},$ be the corresponding scattering matrices. Suppose that $A_{1}(\lambda)=A_{2}(\lambda)$ for every $\lambda\in\mathbb{R}\setminus\{0\}.$ Then there exists a diffeomorphism $\Psi:X_1\rightarrow X_2$, smooth up to $M$, such that
\begin{equation}\label{diffeo-global}
\Psi=Id \text{   on } M \quad \text{and } \quad \Psi^{*}g_2=g_1.
\end{equation}
\end{thm}

This problem is related to the question of reconstructing a compact Riemannian manifold with boundary from the Dirichlet-to-Neumann map for the wave equation that was first solved by  Belishev and Kurylev \cite{BK},  using the Boundary Control Method, and a unique continuation result later proved by Tataru \cite{TAT}. Different proofs, which also rely on the result of Tataru, were given in \cite{KaKuLa,KuLa1}.   The main idea of the proof of Theorem \ref{inverse-full} was to study the scattering matrix in terms of the wave equation, using Friedlander radiation fields,
  and adapt the Boundary Control Method of Belishev and Kurylev \cite{BK} and Tataru \cite{TAT} to this setting.

Our goal in this paper is to prove the analogue of Theorem \ref{inverse-full} when $\mca(\la)$ is known only on an open set of $\Gamma \subset \p X.$  We define the restriction of $\mca(\la)$ to $\Gamma$ as the operator
\begin{gather}
\begin{gathered}
\mca_{\Gamma}(\la):  C_0^\infty(\Gamma) \longrightarrow C^\infty(\Gamma) \\
f \longmapsto (\mca(\la) f)\restr_{\Gamma}.
\end{gathered}\label{agamma}
\end{gather}

Our main result is  the following
\begin{thm}\label{INV}
Let $(X_{1},g_1)$ and $(X_2,g_2)$ be asymptotically hyperbolic manifolds and suppose there exists an open set  $\Gamma\subset(\partial X_{1}\cap \p X_2)$ (as manifolds)  such that $\p X_j\setminus \Gamma$ do not have empty interior, for $j=1,2.$  Let $x_j\in C^{\infty}(X_j)$, $j=1,2$, be a defining function of $\partial X_{j}$ for which \eqref{metric} holds, and let $\mca_{j,\Gamma}(\lambda)$, $j=1,2$, $\lambda\in\mathbb{R}\setminus\{0\},$ be the corresponding scattering matrices restricted to $\Gamma.$ Suppose that $\mca_{1,\Gamma}(\lambda)=\mca_{2,\Gamma}(\lambda)$ for every $\lambda\in\mathbb{R}\setminus\{0\}.$ Then there exists a diffeomorphism $\Psi:X_1\rightarrow X_2$, smooth up to $\partial X_1$, such that
\begin{equation}\label{diffeo}
\Psi=Id \text{   on } \Gamma\quad \text{and } \quad \Psi^{*}g_2=g_1.
\end{equation}
\end{thm}

The  reconstruction of a compact manifold in the case where the Dirichlet-to-Neumann map is only known on part of the boundary was carried out by  Kurylev and Lassas \cite{KuLa} using a modification of the Boundary Control Method. As in the proof of Theorem \ref{inverse-full},  we will adapt the Boundary Control Methods to this setting by using the dynamical formulation of the scattering matrix in terms of Friedlander radiation fields.

Equation \eqref{l2decomp} gives a decomposition of  $L^2(X)$ 
\begin{gather*}
L^2(X)=L^2_{pp}(X) \oplus L^2_{ac}(X),
\end{gather*}
where $L^2_{pp}(X)$ is the finite dimensional space spanned by the eigenfunctions of $\Delta_g$ and $L^2_{ac}(X)$ is the orthogonal complement of $L^2_{pp}(X).$
It follows from Corollary 6.3 of \cite{sa} that
\begin{gather*}
C_0^\infty(\intx) \cap L^2_{ac}(X) \text{ is dense in } L^2_{ac}(X).
\end{gather*}

 Let $u$ satisfy the wave equation
\begin{equation}\label{wave}
\begin{gathered}
\left(D_{t}^2-\Delta_{g}+\frac{n^2}{2}\right)u(t,z)=0\quad \text{on } \mathbb{R}_{+}\times \intx,\\
u(0,z)=f_1(z),\quad D_{t}u(0,z)=f_2(z), \quad f_1, f_2 \in C_0^\infty(\intx).
\end{gathered}
\end{equation}
This equation has a conserved energy given by
\begin{equation}
\begin{gathered}
E(u,\p_t u)(t)=  \int_X \left( |d u(t)|^2- \frac{n^2}{4} |u(t)|^2 + |\p_t u(t)|^2\right) \dvol_g, \\
E(u,\p_t u)(0)=E(f_1,f_2)=\int_X \left( |d f_1|^2- \frac{n^2}{4} |f_1|^2 + |f_2|^2\right) \dvol_g,
\end{gathered}
\end{equation}
$E(f_1,f_2)$ is positive only when projected onto $L^2_{ac}(X).$  As in \cite{sa}, we define
the energy space
\begin{gather*}
H_E(X)=\{ (f_1,f_2):  f_1, f_2 \in L^2(X), \;\ df_1 \in L^2(X) \text{ and }  E(f_1,f_2)<\infty\}
\end{gather*}
and the projector 

\begin{align*}
\mcp_{ac}: L^2(X) \longrightarrow L^2_{ac}(X) \\
 f \longmapsto f-\sum_{j=1}^N \lan f,\phi_j\ran \phi_j,
\end{align*}
where $\{\phi_j, 1\leq j \leq N\}$ are the eigenfunctions of $\Delta_g.$ Let
\begin{gather*}
E_{ac}(X)= \mcp_{ac}(H_E(X)).
\end{gather*}
The wave group induces a strongly continuous group of unitary operators:
\begin{gather*}
U(t): E_{ac}(X) \longrightarrow E_{ac}(X) \\\
(f_1,f_2) \longmapsto (u(t), \p_t u(t)).
\end{gather*}

Following Friedlander \cite{fried0,fried1}, the radiation fields for asymptotically hyperbolic manifolds were defined in \cite{sa}. It was shown in \cite{sa} that  if $x$ is a boundary defining function of $\p X$ for which \eqref{metric} holds, and $(f_1,f_2)\in C_0^\infty(X) \cap E_{ac}(X),$
 then
\begin{gather}
 V_+(x,s,y)= x^{-n/2}u(s-\log x,x,y) \in C^\infty([0,\eps)_x \times \mr_s \times \p X), \label{defv+}
 \end{gather}
and the forward radiation field is defined as the map
\begin{equation}
\begin{gathered}
\mathcal{R}_{+}:C^{\infty}_{0}(\overset{\circ}{X})\times C^{\infty}_{0}(\overset{\circ}{X})\longrightarrow C^{\infty}(\mathbb{R}\times \partial X),\\
\mathcal{R}_{+}(f_1,f_2)(s,y)=D_sV_+(0,s,y)=\lim_{x\downarrow 0} x^{-\frac{n}{2}} D_s u(s-\log x,x,y).
\end{gathered}\label{forad}
\end{equation}
Similarly, if one considers the behavior of $u$ for $t<0,$ again with the initial data 
$(f_1,f_2)\in C_0^\infty \cap E_{ac}(X),$ then  
\begin{gather*}
 V_-(x,s,y)= x^{-n/2}u(s+\log x,x,y) \in C^\infty([0,\eps)_x \times \mr_s \times \p X),
 \end{gather*}
 and thus defines the backward radiation field
\begin{equation}
\begin{gathered}
\mathcal{R}_{-}:C_{0}^{\infty}(\overset{\circ}{X})\times C_{0}^{\infty}(\overset{\circ}{X})\longrightarrow C^{\infty}(\mathbb{R}\times \partial X),\\
\mathcal{R}_{-}(f_1,f_2)(s,y)=D_{s}V_{-}(0,s,y).
\end{gathered} \label{backwrad}
\end{equation}
It was shown in \cite{sa} that the maps $\mcr_{\pm}$ extend to unitary operators 
\begin{gather}
\begin{gathered}
\mcr_{\pm}: E_{ac}(X) \longrightarrow L^2(\mr \times \p X) \\
(f_1,f_2) \longmapsto \mcr_{\pm}(f_1,f_2),
\end{gathered} \label{unitary}
\end{gather}
which are translation representations of the wave group as in the Lax-Phillips theory \cite{lp}, i.e.
\begin{gather}
\mcr_{\pm}(U(T)(f_1,f_2))(s,y)= \mcr_{\pm}(f_1,f_2)(s+T,y).\label{translrep}
\end{gather}

One can define the scattering operator
\begin{gather}
\begin{gathered}
\mathcal{S}:L^2(\mathbb{R}\times\partial X)\longrightarrow L^2(\mathbb{R}\times\partial X),\\ \mathcal{S}=\mathcal{R}_{+}\circ\mathcal{R}_{-}^{-1},
\end{gathered}\label{scatmat2}
\end{gather}
which is unitary in $L^2(\partial X\times \mathbb{R})$ and  commutes with translations in view of \eqref{translrep}. The results of \cite{JS} and \cite{Gui1} were used in  \cite{sa} to show that the stationary and dynamical definitions of the scattering matrix are equivalent.  If  $\mca(\la)$ is defined by
\eqref{scatmat1} and $\mcs$ is defined by  \eqref{scatmat2}, then
\begin{gather}
\mca=\mathcal{F}\mathcal{S}\mathcal{F}^{-1},\label{relation}
\end{gather}
where $\mcf$ denotes the partial Fourier transform in the variable $s.$

The relationship between $\mcr_\pm$ and $\mca(\la)$ was made more explicit in \cite{sa}, where it was shown (see equation (6.12) of \cite{sa}) that if one takes the partial Fourier transform of $\mcr_{+}$ with respect to the variable $s,$   then if one denotes
\begin{gather} 
\mcf(\mcr_+(0,h))(\la, y)= i\la \int_{X} E\left(\novt +i\la, y, z\right) h(z) \dvol_{g}(z), \label{defnE}
\end{gather}
then for any $ f \in C^\infty(\p X),$ 
\begin{gather}
u(z)=\int_{\p X} E\left( \novt+i\la,y',z\right) f(y') \; \dvol_{h_0}(y'), \label{transE1}
\end{gather}
satisfies
\begin{gather}
\begin{gathered}
(\Delta_g-\nsq-\la^2) u=0, \\
u(x,y) = x^{\novt +i\la} F_+ + x^{\novt-i\la} F_-, \text{ where }  F_+\restr_{\p X}=f, \;\  F_-\restr_{\p X}=\mca(\la)f .
\end{gathered}\label{transE2}
\end{gather}
 
\subsection{Acknowledgements}
\label{sec:acknowledgements}
Hora and S\'a Barreto are grateful to the NSF for the support provided under grant DMS-0901334.

\section{The Local Support Theorem}

 The purpose of this section is to establish a relationship between the support of $\mcr_+(0,f)$ and the support of $f.$  Again, it is useful to recall the analogy with the Dirichlet-to-Neumann map on compact manifolds with boundary. If $(X,g)$ is a compact Riemannian manifold with boundary, the finite speed of propagation for the wave equation implies that if $\Gamma\subset \p X$ is an open subset of the boundary and $u(t,z)$ is a solution to the Dirichlet problem for the wave equation
\begin{gather*}
(\p_t^2 +\Delta_g)u=0 \text{ in } (-T,T)\times \intx  \\
u(0,z)=0, \;\  \p_t u(0,z)=f(z) \in C^\infty(\intx), \\
u(t,z)|_{(-T,T)\times \p X}=0
\end{gather*}
then the normal derivative $\p_\mu u(t,z)|_{(-T',T')\times \p X}=0,$ provided  $T'\in (0,T),$ $z\in \overline{\Gamma}$ and $d_g(\supp f, \overline{\Gamma})<T',$ where $d_g$ denotes the distance with respect to the metric $g.$  The converse of this result holds, and is a consequence of a theorem of Tataru \cite{TAT}, see also
\cite{hormander-pa,robbianozuilly,tataru1}. 

 In the current setting, the manifold is not compact and the distance from a point in the interior of $X$ to the boundary is infinite.    If  $x$ is chosen such that the metric $g$ is in the form \eqref{metric}, the corresponding Lorentzian metric takes the form
\begin{gather*}
G_L= dt^2-\frac{dx^2}{x^2} -\frac{h(x,y,dy)}{x^2}.
\end{gather*}

From the construction of the defining function $x$ for which \eqref{metric} holds from \cite{Gra} (see section \ref{intro}),  given two points along the curves normal to $\Gamma,$  the distance between $(x,y)$ and $(\alpha,y),$ with $\alpha<x$  small enough,  is $-\log (\frac{x}{\alpha}),$   and therefore,  if time measures the arc-length along these geodesics, one has $s=t+\log x=\log \alpha.$   In the case of data given on the whole boundary, the following support theorem was proved in  \cite{sa}:
\begin{thm}\label{support-full} A function $f \in L^2_{ac}(X)$ is such that  $\mcr_{+}(0,f)(s,y)=0$ for every $s\leq s_0 <<0$  and $y \in \p X$ if and only if  $f(x,y)=0$ if $x\leq e^{s_0},$ $y \in \p X.$
\end{thm}
One should remark that the fact this result is true for $f\in L^2_{ac}(X)$ is particular to the hyperbolic setting. For instance, the analogue of this result is not true on Euclidean space; in that case one needs to assume that $f$ is rapidly decaying, i.e. $f$ vanishes to infinite order at $\p X.$  In the case where $(X,g)$ is the hyperbolic space, this result is due to Lax and Phillips \cite{lp1,lp2,lp3}, see also the work of Helgason \cite{helgason1,helgason2,helgason3,helgason4}.  Our goal  for this section is to prove the local version of this result. 
First we observe that if $\Gamma \subset \p X$ is an open subset, the construction of the defining function works just as well if one restricts it to  $\Gamma.$ 
We will prove
\begin{thm}\label{support-local}     Let $\Gamma\subset \p X$ be an open subset and $s_0\in \mr.$ A function  $f \in L^2_{ac}(X)$ is such that  $\mcr_{+}(0,f)(s,y)=0$ for every $s\leq s_0$ and $y\in \ogamma$ if and only if  there exists $\eps>0$ such that $f=0$ a.e. in the set
 \begin{gather}
\mcd_{s_0}(\Gamma)=\{z\in X:  \exists \; w=(\alpha,y') \text{ with }  0<\alpha <\eps \leq e^{s_0} \text{ and } y'\in \Gamma, \;\ d_g(z,w) < \log(\frac{e^{s_0}}{\alpha})\}, \label{mcdnew}
\end{gather}
where $d_g$ denotes the distance function with respect to the metric $g.$
\end{thm}

Notice that  $\mcd_{s_0}(\Gamma)$ is equal to the union of open balls centered at $(\alpha,y'),$ with  $0<\alpha<\eps,$ and $y' \in \Gamma,$ and radii 
$\log(\frac{e^{s_0}}{\alpha})$ measured with respect to the metric $g.$ Therefore $\mcd_{s_0}(\Gamma)$ is an open subset of $\intx.$

 In Theorem \ref{support-full} one assumes that $s_0<<0$ and the conclusion about the support of $f$ is expressed in terms of coordinates $(x,y),$ which are only valid in a neighborhood of $\p X.$  There is no restriction on $s_0$ in Theorem \ref{support-local};  so it is a generalization of Theorem \ref{support-full} in this sense as well.   On the other hand, if $\Gamma=\p X$ and $s_0<<0$ is such that coordinates $(x,y)$ hold in $(0,\eps) \times \p X,$ with $e^{s_0}<\eps,$  then $\mcd_{s_0}(\p X)=\{(x,y): x< e^{s_0}, \; y\in \p X\}.$ Indeed, given a point $z=(x,y)$ with $x<e^{s_0},$ if $0<\alpha<x,$ $d_g((\alpha,y), (x,y))=\log(\frac{x}{\alpha})< \log(\frac{e^{s_0}}{\alpha}).$  But, if $\eps> x>e^{s_0},$ and $z=(x,y'),$ then for any $w=(\alpha,y),$ with $\alpha<e^{s_0},$ $d_g(z,w)> \log(\frac{e^{s_0}}{\alpha}).$ So 
 $z\not \in \mcd_{s_0}(\p X).$

Lax and Phillips  \cite{lp4} also proved this result when $(X,g)$ is the hyperbolic space.  It is useful to explain what the set $\mcd_{s_0}(\Gamma)$  is when $(X,g)$ is the hyperbolic space.   It is easier to do the computations for the half-space model of hyperbolic space which is given by
\begin{gather*}
\mh^{n+1}=\{ (x,y): \;\ x>0, \;\ y\in \mr^n\}, \text{ and the metric } g=\frac{dx^2+dy^2}{x^2}.
\end{gather*}
The distance function between $z=(x,y)$ and $w=(\alpha,y')$ satisfies
\begin{gather*}
\cosh d_g(z,w)= \frac{x^2+{\alpha}^2+|y-y'|^2}{2x\alpha}.
\end{gather*}
Since $d_g(z,z')\leq \log(\frac{e^{s_0}}{\alpha}),$ we obtain
\begin{gather*}
\left( x-\ha e^{s_0}(1+\alpha^2e^{-2s_0})\right)^2+|y-y'|^2\leq \oq e^{2s_0}(1+\alpha^2 e^{-2s_0})^2-\alpha^2=\oq e^{2s_0}(1-\alpha^2 e^{-2s_0})^2,
\end{gather*}
which corresponds to a ball $D(\alpha)$ centered at $(\ha e^{s_0}(1+\alpha^2e^{-2s_0}),y')$ and radius $\ha e^{s_0}(1-\alpha^2e^{-2s_0}).$ Since $\alpha<e^{s_0},$ we  have $D(\alpha)\subset D(0),$ as shown in figure \ref{fig6h}.  This ball is tangent to the plane $x=e^{s_0}$ at the point $(e^{s_0},y').$    When $\alpha=0$ the ball $D(0)$  has center $(\ha e^{s_0}, y')$ and radius $\ha e^{s_0}$ and is also tangent to the plane $\{x=0\}.$ The boundary of $D(0)$ is called a horosphere since it is orthogonal to the geodesics emanating from the point $(0,y').$   When $\alpha=e^{s_0},$ $D(e^{s_0})=(e^{s_0},y').$

The set $\mcd_{s_0}(\Gamma)$ consists of the union of horospheres with radii $\ha e^{s_0}$ tangent to points $(0,y')$ with $y'\in \Gamma,$ see figure \ref{fig7}.  In the case of $\mh^{n-1},$ the radiation field is given in terms of the horocyclic Radon transform.   The result of Lax and Phillips says that if the integral of $f$ over all horospheres  tangent to points $(0,y),$ with $y\in \Gamma$ and radii less than or equal to $\ha e^{s_0},$ then $f=0$ in the region given by the union of these horocycles.
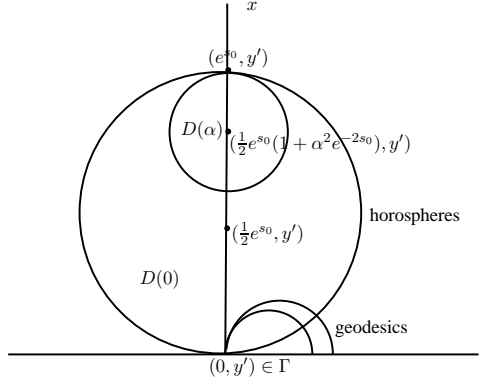
\begin{figure}
% Generated with LaTeXDraw 2.0.8
% Fri May 31 14:49:40 EDT 2013
% \usepackage[usenames,dvipsnames]{pstricks}
% \usepackage{epsfig}
% \usepackage{pst-grad} % For gradients
% \usepackage{pst-plot} % For axes
\scalebox{.65} % Change this value to rescale the drawing.
{
\begin{pspicture}(0,-3.58)(9.64,4.22)
\psline[linewidth=0.04cm](0.0,-3.1)(9.62,-3.1)
\psdots[dotsize=0.12](4.48,-0.52)
\psdots[dotsize=0.12](4.5,2.72)
\usefont{T1}{ptm}{m}{n}
\rput(4.71,2.965){$(e^{s_0}, y')$}
\psline[linewidth=0.04cm](4.48,2.72)(4.44,-3.1)
\psarc[linewidth=0.04](5.33,-3.09){0.89}{0.0}{180.0}
\psarc[linewidth=0.04](5.55,-3.09){1.09}{0.0}{175.03026}
\usefont{T1}{ptm}{m}{n}
\rput(5.31,-0.635){$(\ha e^{s_0},y')$}
\usefont{T1}{ptm}{m}{n}
\rput(6.39,1.205){$(\ha e^{s_0}(1+\alpha^2e^{-2s_0}),y')$}
\psdots[dotsize=0.12](4.5,1.46)
\pscircle[linewidth=0.04,dimen=outer](4.51,1.45){1.23}
\pscircle[linewidth=0.04,dimen=outer](4.34,-0.2){2.9}
\usefont{T1}{ptm}{m}{n}
\rput(3.97,1.485){$D(\alpha)$}
\usefont{T1}{ptm}{m}{n}
\rput(3.1,-1.575){$D(0)$}
\usefont{T1}{ptm}{m}{n}
\rput(8.33,-0.275){horospheres}
\usefont{T1}{ptm}{m}{n}
\rput(7.44,-2.575){geodesics}
\usefont{T1}{ptm}{m}{n}
\rput(4.94,-3.355){$(0,y')\in \Gamma$}
\psline[linewidth=0.04cm](4.48,2.64)(4.48,4.08)
\usefont{T1}{ptm}{m}{n}
\rput(4.99,4.025){$x$}
\end{pspicture} 
}
\caption{The horospheres tangent at $(0,y')$ and the balls $D(\alpha)$}
\label{fig6h}
\end{figure}
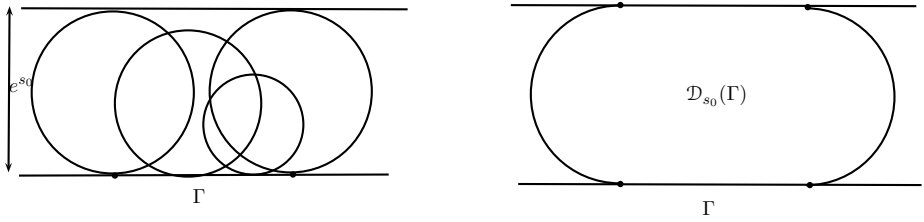
\begin{figure}
% Generated with LaTeXDraw 2.0.8
% Fri May 31 18:02:00 EDT 2013
% \usepackage[usenames,dvipsnames]{pstricks}
% \usepackage{epsfig}
% \usepackage{pst-grad} % For gradients
% \usepackage{pst-plot} % For axes
\scalebox{.70} % Change this value to rescale the drawing.
{
\begin{pspicture}(0,-2.07)(17.88,2.05)
\psline[linewidth=0.04cm](0.72,-1.27)(7.74,-1.25)
\psdots[dotsize=0.12](2.54,-1.27)
\psdots[dotsize=0.12](5.92,-1.25)
\usefont{T1}{ptm}{m}{n}
\rput(4.14,-1.665){$\Gamma$}
\pscircle[linewidth=0.04,dimen=outer](2.5,0.31){1.56}
\pscircle[linewidth=0.04,dimen=outer](5.88,0.33){1.56}
\psline[linewidth=0.04cm](0.76,1.91)(8.1,1.89)
\psline[linewidth=0.04cm,arrowsize=0.05291667cm 2.0,arrowlength=1.4,arrowinset=0.4]{<->}(0.54,1.95)(0.52,-1.21)
\usefont{T1}{ptm}{m}{n}
\rput(0.77,0.455){$e^{s_0}$}
\pscircle[linewidth=0.04,dimen=outer](3.93,0.1){1.41}
\pscircle[linewidth=0.04,dimen=outer](5.17,-0.3){0.97}
\psline[linewidth=0.04cm](10.2,-1.43)(17.86,-1.45)
\psdots[dotsize=0.12](12.14,-1.43)
\psdots[dotsize=0.12](15.74,-1.45)
\usefont{T1}{ptm}{m}{n}
\rput(13.82,-1.885){$\Gamma$}
\rput{89.76527}(12.352723,-11.857647){\psarc[linewidth=0.04](12.129524,0.2728937){1.6898013}{0.0}{181.34106}}
\rput{270.01706}(15.432061,15.893167){\psarc[linewidth=0.04](15.664981,0.22825392){1.6812352}{0.0}{179.69196}}
\usefont{T1}{ptm}{m}{n}
\rput(13.97,0.255){$\mcd_{s_0}(\Gamma)$}
\psdots[dotsize=0.12](12.14,1.97)
\psdots[dotsize=0.12](15.7,1.93)
\psline[linewidth=0.04cm](10.06,1.97)(17.76,1.95)
\end{pspicture} 
}
\caption{The set $\mcd_{s_0}(\Gamma)$ when $(X,g)$ is the hyperbolic space is given by the union of horospheres tangent to points on $\Gamma$ and radii less than or equal to $\ha e^{s_0}.$}
\label{fig7}
\end{figure}

It is interesting to explain Theorem \ref{support-local} in terms of sojourn times in the case where $(X,g)$ is non-trapping.  In this setting the sojourn time  plays the role of the distance function to the boundary of $X$ and  is closely related to the Busemann function used in differential geometry.  The sojourn times for non-trapping asymptotically hyperbolic manifolds was studied in \cite{sawu}.  Let  $g^*$ denote the principal symbol of the Laplacian with respect to $g.$ In local coordinates 
 \eqref{metric},
 \begin{gather*}
 g^*= x^2 \xi^2 + x^2 h(x,y,\eta),
 \end{gather*}
 where $h$ is the principal symbol of $\Delta_h$ as a differential operator on $\p X.$    The set of points 
 $$S^* \intx= \{ q=(z,\zeta) \in T^*\intx : g^*(q)=1\}$$
  is the unit co-sphere bundle in $T^*X$ with respect to $g^*.$ A integral curve of  the Hamilton vector field of $g^*$ is called a  bicharacteristic. These curves are parametrized  by time $t$ and are denoted by  $\exp(t H_{g^*})( q),$ and since $H_{g^*}$ is tangent to the level surfaces of $g^*,$  $t$ is the arc-length. It is well known that if  $\pi_1$ denotes the canonical projection 
 $\pi_1: T^* \intx \longrightarrow \intx,$ then $\pi_1(\gamma(t))$ is a geodesic of the metric $g$ in $X$ passing through $z,$ the projection of the point $q.$  We say that a bicharacteristic is not trapped for positive (negative) times, if its projection to $\intx$ leaves any compact set $K \subset \intx$ in finite time as $t\rightarrow \infty (-\infty).$  The manifold $(X,g)$ is non-trapping if every bicharacteristic is non-trapped for positive and negative times.    In this case, it was shown in \cite{sawu} that the following functions are well defined
 \begin{gather*}
 s(q)= \lim_{t\rightarrow \infty} (t+ \log[ x(\exp (t H_{g^*})(q))]) \text{ and } \\
 y=   \lim_{t\rightarrow \infty}  y(\exp (t H_{g^*})(q)),
 \end{gather*}
 where $x(\exp (t H_{g^*})(q))$ and $y(\exp (t H_{g^*})(q))$ denote the coordinates \eqref{metric} of the point  $\pi_1(\gamma(t)).$   The function $s(q)$ is called the sojourn time
 of the bicharacteristic through $q.$
 In a compact manifold, this would be the analogue of the set of points in the interior whose distance to the boundary is $s.$    We have the following consequence of 
 Theorem \ref{support-local}:
 \begin{cor}\label{support-local1} Let $f$ and $\Gamma\subset \p X$  satisfy the hypotheses of Theorem \ref{support-local} and suppose that $(X,g)$ is non-trapping. Then $f=0$ a.e on the set of points $z\in \intx$ such that  exists a geodesic  $\gamma(t)$ parametrized by the arc-length such that $\gamma(t) \rightarrow y\in \Gamma$ as $t\rightarrow \infty,$ and
 \begin{gather*}
 \lim_{t\rightarrow \infty}( t+\log (x(\gamma(t)))=s<s_0.
 \end{gather*}
 \end{cor}
 \begin{proof} Suppose there exists a geodesic $\gamma(t),$ parametrized by the arc-length $t$ such that $\gamma(0)=z$ and $\lim_{t\rightarrow \infty} \gamma(t)=y,$ moreover $\lim_{t\rightarrow \infty}(t+\log(x(\gamma(t)))=s<s_0.$  Since $t$ is the arc-length parameter,
$d(z, (x(\gamma(t)),y)) \leq t$ and $s<s_0,$ then there exists $T>0$ such that for $t>T,$  $\gamma(t)\in U\sim [0,\eps) \times \p X,$ where coordinates \eqref{metric} are valid and $t+\log x(\gamma(t))<s_0.$   Therefore, if $t>T,$
\begin{gather*}
d(z, (x(t),y)) \leq t < s_0-\log x(\gamma(t))= \log(\frac{e^{s_0}}{x(\gamma(t))}).
\end{gather*}
Hence $z\in \mcd_{s_0}(\Gamma).$  
\end{proof}
The proof of Theorem \ref{support-local} will be divided in several steps.  We begin by proving the sufficiency of the condition in Theorem \ref{support-local}.
\begin{lemma} Let  $f \in L^2_{ac}(X)$ be such that $f(z)=0$ in the set $\mcd_{s_0}(\Gamma).$    Then  $\mcr_+(0,f)(s,y)=0$ if $s\leq s_0$ and $y \in \Gamma.$
\end{lemma}
\begin{proof}  Let $u(t,w)$ satisfy the wave equation 
\eqref{wave} with initial data $(0,f).$  The finite speed of propagation for solutions of the wave equation guarantees that $u(t,w)=0$ if  $0\leq t \leq d_g(w,\supp f).$ In particular, since $f(z)=0$ for $z\in \mcd_{s_0}(\Gamma),$  if $w=(\alpha,y)$ with $y\in \Gamma,$ then
$u(t,w)=0$ if $0\leq t \leq \log \left(\frac{e^{s_0}}{\alpha}\right).$   Since $s=t+\log x,$ when $x=\alpha$ we have that
$V_+(\alpha,s,y)=x^{-\novt} u(s-\log \alpha, \alpha,y)=0$ provided $ 0\leq s-\log \alpha \leq \log \left(\frac{e^{s_0}}{\alpha}\right).$  Therefore one has
$V_+( \alpha,s,y)=0$ provided $\log \alpha\leq s\leq s_0$ and  $y \in \Gamma.$  This implies that $\mcr_+(0,f)(s,y)=0$ if $s\leq s_0$ and $y \in \Gamma.$
\end{proof}

The proof of the converse relies on delicate unique continuation results.  First, it is important to realize that  we may assume that $f \in C^\infty(\intx).$  Indeed, since $\mcr(0,f)(s,y)=0$ for $s\leq s_0$ and $y\in \Gamma,$ we may take the convolution of  $\mcr_+(0,f)$ with $\psi_\del\in C_0^\infty(\mr),$ even and supported in $(-\del,\del),$ with $\int \psi_\del(s) \; ds=1,$ and hence $H_\del(s,y)=\psi_\del*\mcr_+(0,f)=\mcr(0,\widetilde{f_\del})=0$ if $s\leq s_0+\del,$ and since for every $k\geq 0,$
\begin{gather*}
\p_s^{2k} H_\del(s,y)=\mcr_+(0, (\Delta-\frac{n^2}{4})^k \widetilde{f_\del}) \in L^2(\mr\times \p X),
\end{gather*}
and using that $\mcr_+$ is unitary, then $(\Delta-\frac{n^2}{4})^k \widetilde{f_\del} \in L^2(X).$ Therefore, by elliptic regularity $\widetilde{f_\del}\in C^\infty(\intx).$
If one proves that $\widetilde{f_\del}(z)=0$ for $z\in \mcd_{s_0}(\Gamma),$ is supported in $\{x\geq e^{s_0+\del}\},$  since $\widetilde{f_\del}\rightarrow f$ as $\del\rightarrow 0,$ it follows that 
$f(z)=0$ for $z\in \mcd_{s_0}(\Gamma).$  

The next step in the proof is
\begin{prop}\label{step-1}  Let $f\in C^\infty(X)$ satisfy the hypotheses of Theorem \ref{support-local}.    Let  $u$ satisfy the initial value problem for the wave equation \eqref{wave} with initial data $(0,f),$ and let $V_+(x,s,y)$ be defined as in \eqref{defv+}.  Then, in the sense of distributions,  $\p_x^k V_+(0,s,y)=0,$ $k=0,1,...,$ provided $s\leq s_0,$ and
$y\in \Gamma.$ Moreover, for every $p\in \Gamma$ there exists $\del>0$ such that $V_+(x,s,y)=0$ if $0<x<\del,$ $|y-p|<\del$ and $s<\log \del.$
\end{prop}
\begin{proof} In local coordinates \eqref{metric} and for $s=t+\log x,$ $x\in [0,\eps),$ the wave operator, conjugated by appropriate powers of 
$x,$ can be written as
\begin{gather}
P=-x^{-\frac{n}{2}-1} \left(D_t^2-\Delta-\frac{n^2}{4}\right) x^{\novt}= \p_x(2\p_s+x \p_x) - x\Delta_h + A\p_s + A x\p_x +\novt A,\label{operatorP}
\end{gather}
where $\Delta_h$ is the Laplace operator on $\p X$ corresponding to the metric $h(x).$ In local coordinates
\begin{gather}
\begin{gathered}
\Delta_h=-\frac{1}{\sqrt{\theta}} \p_{y_i}( \sqrt{\theta} \; h^{ij} \p_{y_j}) \text{ where } \\
 h=(h_{ij}(x,y)) \;\ h^{-1}=(h^{ij}(x,y)), \;\  \theta=\det(h_{ij}) \text{ and } A=\frac{1}{\sqrt{\theta}}\p_{x} \sqrt{\theta}.
\end{gathered}\label{deftheta}
\end{gather}
The Cauchy problem \eqref{wave}, with initial data $(0,f)$ translates into the following initial value problem for $V(x,s,y)= x^{-\novt} u(s+\log x,x,y),$
\begin{equation}\label{modCP}
\begin{gathered}
PV(x,s,y)=0 \text{ in } \mr \times (0,\eps)_x \times \p X,\\
V(x,\log x, y)=0, \;\ D_s V(x,\log x,  y)=x^{-\novt} f(x,y).
\end{gathered}
\end{equation}
Since one cannot prove unique continuation results across $x=0,$ then as in \cite{sa} we have to compactify the space in a suitable way, and instead of working with coordinates $x$ and $s,$ it is more convenient to work with the variables
\begin{gather*}
s_+=s=t+\log x \text{ and } s_-=t-\log x.
\end{gather*}
Since we are interested in the behavior of $V_+(x,s,y)$ defined in \eqref{defv+}  for $s=s_+\sim -\infty,$ and by parity for $s_-\sim \infty,$ we introduce the following change of variables
\begin{gather}
 \mu=e^{-\frac{s_-}{2}} \text{ and } \nu=e^{\frac{s_+}{2}}. \label{defmunu}
 \end{gather}
This implies that
\begin{gather*}
s=2\log \nu, \;\ x=\mu \nu. 
 \end{gather*}
 We remark that the  change of variables $(t,x,y) \mapsto \left(\frac{\mu+\nu}{2}, \frac{\mu-\nu}{2},y\right),$ which will be used below, plays the role of the Kelvin transform for the Euclidean wave equation.   
 
In coordinates $(\mu,\nu,y),$ the operator $P$  has the form
 \begin{gather*}
 \wtp=\p_\mu\p_\nu - \mu\nu \Delta_h + \ha A(\mu\p_\mu+\nu\p_\nu) +\novt A,
 \end{gather*}
where $h=h(\mu\nu),$ $A=A(\mu\nu,y).$  If $W(\mu,\nu,y)= V_+(\mu\nu, 2\log \nu,y),$ the Cauchy problem \eqref{modCP} becomes
\begin{equation}\label{modCP1}
\begin{gathered}
\wtp W=0, \;\ \ \mu, \nu \in (0,\eps), \;\ y \in \p X \\
W(\mu,\mu,y)=0, \;\ \p_\mu W(\mu,\mu,y)=- \mu^{-1-n} f(\mu^2,y).
\end{gathered}
\end{equation}

Recall that we are assuming that $f\in C^\infty(\intx),$ so $W$ is $C^\infty$ in the region $\{\mu>0, \nu>0\}.$ The issue here is the behavior of $W$ at $\{\mu=0\}\cup \{\nu=0\}.$
One should notice that if $F(\mu,y)=\mu^{-1-n} f(\mu^2,y),$ then
\begin{gather}
\int_0^\eps \int_{\p X} \mu |F(\mu,y)|^2 \; \theta^\ha(\mu^2,y) dy d\mu=\ha \int_0^{\eps^2} \int_{\p X}  |f(x,y)|^2 \; x^{-n-1} \theta^\ha(x,y) dx dy=\ha ||f||_{L^2(X)}^2.\label{l2nf}
\end{gather}

We know from Theorem 2.1 of \cite{sa} that if $f \in C_0^\infty(\intx)\cap  L^2_{ac}(X),$ then $W$ has a $C^\infty$  extension $W(\mu,\nu,y) \in C^\infty([0,\eps]\times [0, \eps] \times \p X)$ up to 
$\{\mu =0\} \cup \{\nu=0\},$ and since  $\p_s=\ha(\nu\p_\nu-\mu\p_\mu),$ then, provided $f \in C_0^\infty(\intx)\cap  L^2_{ac}(X),$
\begin{gather}
\mcr_+(0,f)(2\log \nu,y)= \ha\left[(\nu\p_\nu- \mu\p_\mu) W(\mu,\nu,y)\right]|_{\mu=0}= \ha  \nu \p_\nu W(0,\nu,y), \label{newradf}
\end{gather}
and we want to show that this restriction makes sense for $f\in L_{ac}^2(X).$  We will work in the region $\{\nu \geq \mu\},$ but since the solution to \eqref{modCP1} is  odd under the change
$(\mu,\nu)\mapsto (\nu,\mu),$ a similar analysis works for the backward radiation field in the region $\{\nu \leq \mu\}.$

Again, assuming that $f\in C_0^\infty(\intx)\cap L^2_{ac}(X),$ and $W$ satisfies \eqref{modCP1}, one can combine equations (4.11), (4.14) and (4.15) of \cite{sa}, and \eqref{l2nf}  to arrive at the following estimate (see figure \ref{figure1}):  For $\mu_0\in [0,\eps), $
$T\in [0,\eps),$ there exists $C>0$ depending on the operator $\wtp$ and $\eps$ only, such that 
\begin{equation}\label{enerest}
\begin{gathered}
I(W,\mu_0,T)=\int_{\mu_0}^T \int_{\p X} \left.\left[ (|W|^2+\mu|\p_\mu W|^2 + \mu\nu^2 |d_{h(\mu \nu)} W|^2)\sqrt{\theta}(\mu \nu)\right] \right|_{\nu=T}\, dy d\mu+ \\
\int_{\mu_0}^T \int_{\p X} \left.\left[ (|W|^2+\nu|\p_\nu W|^2 + \mu^2\nu |d_{h(\mu \nu)} W|^2) \; \sqrt{\theta}(\mu \nu)\right]\right|_{\mu=\mu_0} \,  dy d\nu \leq C ||f||_{L^2(X)},
\end{gathered}
\end{equation}
\begin{figure}
% Generated with LaTeXDraw 2.0.8
% Thu Mar 07 10:53:44 EST 2013
% \usepackage[usenames,dvipsnames]{pstricks}
% \usepackage{epsfig}
% \usepackage{pst-grad} % For gradients
% \usepackage{pst-plot} % For axes
\scalebox{.7} % Change this value to rescale the drawing.
{
\begin{pspicture}(0,-3.3488476)(8.6610155,3.3288476)
\psline[linewidth=0.04cm](0.96101564,3.3088477)(0.9410156,-2.6911523)
\psline[linewidth=0.04cm](0.9410156,-2.7111523)(8.641016,-2.7111523)
\usefont{T1}{ptm}{m}{n}
\rput(7.662471,-3.1061523){$\mu$}
\usefont{T1}{ptm}{m}{n}
\rput(0.3724707,2.8738477){$\nu$}
\psline[linewidth=0.04cm](0.9410156,-2.7311523)(6.4210157,2.5088477)
\psline[linewidth=0.04cm](2.2010157,1.2688477)(2.2010157,-1.4911523)
\psline[linewidth=0.04cm](2.2010157,1.2688477)(5.1810155,1.2888477)
\psline[linewidth=0.04cm,linestyle=dashed,dash=0.16cm 0.16cm](0.9810156,1.2888477)(2.2210157,1.2888477)
\psline[linewidth=0.04cm,linestyle=dashed,dash=0.16cm 0.16cm](2.2210157,-1.4911523)(2.1810157,-2.7711523)
\usefont{T1}{ptm}{m}{n}
\rput(2.2224708,-3.1261523){$\mu_0$}
\usefont{T1}{ptm}{m}{n}
\rput(0.4524707,1.2938477){$T$}
\end{pspicture} 
}
\caption{The region of integration in \eqref{enerest}}
\label{figure1}
\end{figure}
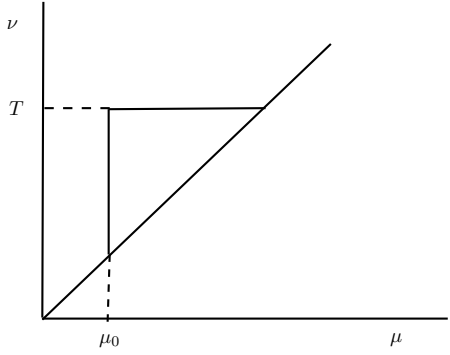

If  $f \in L_{ac}^2(X)$ and if we take a sequence $f_j\in C_0^\infty(\intx)\cap L^2_{ac}(X),$ with $||f-f_j||_{L^2(X)}\rightarrow 0,$  \eqref{enerest} shows that fixed $\mu_0\in [0,\eps_0),$ then
\begin{gather*}
I(W_j-W_k,\mu_0,T) \leq C ||f_j-f_k||_{L^2(X)},
\end{gather*}
and in particular, if  $\mu_0\in [0,\eps),$ and if $W$ is a solution of \eqref{modCP1} with $f\in L^2_{ac}(X),$ then for $\mu_0\in [0,\eps),$
\begin{gather}
\int_{\mu_0}^T \int_{\p X} \nu |\p_\nu W (\mu_0,\nu,y)|^2 \sqrt{\theta}(\mu_0 \nu ,y) d\nu dy \leq C ||f||_{L^2(X)} \label{limitest}
\end{gather}
is well defined. Since the radiation field is unitary, then in the sense of  \eqref{limitest} for $\mu_0=0,$  equation \eqref{newradf}  holds for $f\in L_{ac}^2(X).$

Next we want to show that, if
$\nu\p_\nu W(0,\nu,y)=0,$ and by symmetry $\mu\p_\mu W(\mu,0,y)=0,$ then in the sense of distributions, $W$ has an extension across $\{\mu=0\}\cup\{\nu=0\}$ such that
$W=0$ for $\mu<0,$ $\nu<0,$ with $|\mu|<\eps,$ $|\nu|<\eps,$ which remains a solution to \eqref{modCP1}.   This is possible due to the partial hypoellipticity of the operator $P.$

As it was done in \cite{sa}, it is convenient to get rid of the term $A(\mu\p_\mu+\nu\p_\nu)$ in \eqref{operatorP}, and to achieve this we conjugate the operator by
$\theta^{-\oq}.$  We get that
\begin{gather}
\wtq= \theta^{\oq} \wtp \theta^{-\oq}= \p_\mu\p_\nu - \mu\nu \Delta_h+ \mu\nu B(\mu\nu,y,\p_y)+ C(\mu\nu,y), \label{operatortP}
\end{gather}
where $C$ is $C^\infty$ and $B(\mu\nu,y,\p_y)=\sum_{j=1}^N b_j(\mu\nu,y) \p_{y_j},$  with $b_j$ $C^\infty.$ 
Let  $\wtw=\theta^{\oq} W,$  then $\wtq \wtw=0.$  For $\phi(y) \in C_0^\infty(U),$ with $U\subset \subset \Gamma,$ let 
\begin{gather}
G(\mu,\nu)= \int_{\p X} \wtw(\mu,\nu,y) \phi(y) \; dy \label{defG}
\end{gather}

Let $Z(\mu,\nu,y,D_y)= \mu\nu \Delta_h- \mu\nu B(\mu\nu,y,\p_y)- C(\mu\nu,y),$ and let $Z^*(\mu\nu,y,D_y)$ denote its transpose with respect to the
$L^2(\p X)$ product defined by  \eqref{defG}, then
\begin{gather}
\p_\mu\p_\nu G(\mu,\nu)= \int_{\p X} \wtw(\mu,\nu,y) Z^*(\mu\nu,D_y)\phi(y) \; dy \label{QestG}
\end{gather}

It follows from \eqref{enerest} that  there exists $C>0$ such that
\begin{gather}
\begin{gathered}
\int_0^T  |\p_\mu \p_\nu G(\mu,T)|^2 \; d\mu \leq C( \sum_{|\alpha|\leq 2} \sup |\p_y^\alpha \phi|)^2  ||f||_{L^2(X)}^2, \text{ for  } T\in (0,\eps), \\
\int_{\mu_0}^T  |\p_\mu \p_\nu G(\mu_0,\nu)|^2 \; d\nu  \leq  C( \sum_{|\alpha|\leq 2} \sup |\p_y^\alpha \phi|)^2 ||f||_{L^2(X)}^2, \text{ for  } \mu_0\in (0,\eps).
\end{gathered} \label{regG}
\end{gather}
Let us denote $K= \left(\sum_{|\alpha|\leq 2} \sup |\p_y^\alpha \phi|\right)  ||f||_{L^2(X)}.$  Therefore, if $\del<\mu<\eps,$
\begin{gather*}
\left|\p_\nu G(\mu,\nu)-\p_\nu G(\del,\nu)\right|=\left| \int_\del^\mu \p_s\p_\nu G(s,\nu) \; ds\right|\leq C K  (\mu-\del)^\ha.
\end{gather*}
Hence, for $\nu>0,$
\begin{gather*}
\limsup_{\del\rightarrow 0} |\p_\nu G(\del,\nu)| \leq \liminf_{\mu\rightarrow 0} |\p_\nu G(\mu,\nu)|.
\end{gather*}
Hence, $\lim_{\mu\rightarrow 0} |\p_\nu G(\mu,\nu)|$ exists. On the other hand, since $\mcr_+(0,f)(s,y)=0,$ $y \in \Gamma$ and $s\leq s_0,$  so, according to \eqref{newradf} it follows that 
\begin{gather*}
\p_\nu G(0,\nu)=0, \;\ \nu\in (0, e^{\frac{s_0}{2}}).
\end{gather*}

Now we use \eqref{regG} to show that if $\nu \in (0,e^{\frac{s_0}{2}}),$ then there exists $C>0$
\begin{gather}
|\p_\nu G(\mu,\nu)|= \left|\int_0^\mu \p_\mu\p_\nu G(\mu,\nu) \; d\mu\right| \leq C K \mu^\ha. \label{decay1}
\end{gather}

Since $W(\mu,\mu,y)=0,$  we have for $\mu<\nu,$
\begin{gather}
|G(\mu,\nu)|= |\int_\mu^\nu \p_s G(\mu,s) \; ds| \leq C K \mu^\ha (\nu-\mu)^\ha. \label{bdfG}
\end{gather}
This shows that for every $\phi\in C_0^\infty(U)$
\begin{gather*}
\left|\int_{\p X} \wtw(\mu,\nu,y) \phi(y) \; dy \right| \leq C K \mu^\ha, \\
\left|\int_{\p X} \p_\nu \wtw(\mu,\nu,y) \phi(y) \; dy \right| \leq C K  \mu^\ha.
\end{gather*}

Since $C_0^\infty(\mr^2)\times C_0^\infty(U)$ is dense in $C_0^\infty(\mr^2\times U),$ it follows that for any $\psi(\mu,\nu,y),$ with $|\mu|\leq \eps,$ $|\nu|\leq \eps,$
\begin{gather}
\begin{gathered}
\left|\int_{\p X} \wtw(\mu,\nu,y) \psi(\mu,\nu, y) \; dy \right| \leq C \left(\sum_{|\alpha|\leq 2} \sup |\p_y^\alpha \psi| \right) ||f||_{L^2(X)}\;  \mu^\ha, \\
\left|\int_{\p X} \p_\nu \wtw(\mu,\nu,y) \psi(\mu,\nu,y) \; dy \right| \leq C \left(\sum_{|\alpha|\leq 2} \sup |\p_y^\alpha \psi| \right) ||f||_{L^2(X)}\;  \mu^\ha.
\end{gathered}\label{linftybds}
\end{gather}
Now we differentiate \eqref{QestG} with respect to $\p_\nu.$  We have for $|\mu|<\eps,$ $|\nu|<\eps,$
\begin{gather*}
\p_\nu \p_\mu\p_\nu G(\mu,\nu)= \int_{\p X} \left[ \p_\nu \wtw(\mu,\nu,y) Z^*(\mu\nu,D_y)\phi(y)+ \wtw(\mu,\nu,y) \p_\nu  Z^*(\mu\nu,y)) \phi(y)\right]  \; dy, 
\end{gather*}
and so we obtain from \eqref{linftybds}
\begin{gather*}
|\p_\mu\p_\nu^2 G(\mu,\nu,y)| \leq C (\sum_{|\alpha|\leq 4} |\sup \p_y^\alpha \phi| ) ||f||_{L^2(X)}\;   \mu^\ha
\end{gather*}
Let us denote $K_N(\psi)= \left(\sum_{|\alpha|\leq N} |\sup \p_y^\alpha \phi| \right)||f||_{L^2(X)}.$ Since $\wtw(\mu,\mu,y)=0,$ it follows that $\p_\mu\p_\nu G(\mu,\mu)=0,$ and so we have
\begin{gather}
|\p_\mu\p_\nu G(\mu,\nu)|= \left|\int_\mu^\nu \p_\mu \p_s^2 G(\mu,s) \; ds\right| \leq K_4(\phi) \mu^\ha. \label{2point18}
\end{gather}
On the other hand, since $W(\mu,\mu,y)=0,$ it follows that  $(\p_\mu W)(\mu,\mu,y)=-(\p_\nu W)(\mu,\mu,y).$ In particular, when $\nu=\mu,$ we have
\begin{gather*}
|\p_\mu G(\mu,\mu)|\leq C K_2(\phi) \mu^\ha,
\end{gather*}
and since
\begin{gather*}
\p_\mu G(\mu,\nu)=(\p_\mu G)(\mu,\mu) + \int_\mu^\nu \p_s\p_\mu G(\mu,s) \; ds,
\end{gather*}
 we have
\begin{gather}
|\p_\mu G(\mu,\nu)| \leq C(K_2(\phi) + K_4(\phi)) \mu^\ha. \label{decay20}
\end{gather}

Proceeding as above,  since $\p_\nu G(0,\nu)=0,$ it follows from \eqref{2point18}  that $|\p_\nu G(\mu,\nu)|\leq C K_4(\phi) \mu^\tha,$ and since $G(\mu,\mu)=0,$ then $|G(\mu,\nu)|\leq C K_4(\phi) \mu^\tha,$
and $|\p_\mu \p_\nu^2 G(\mu,\nu)|\leq C K_6(\phi) \mu^\tha.$  So iterating this argument,  and using the symmetry of $W$ we get that for $k\geq 0,$
\begin{gather}
\begin{gathered}
\p_\mu^k G(0,\nu)=0,  \;\  \p_\nu^k G(\mu, 0)=0, \\
|(\p_\mu G)(\mu,\mu)|= |(\p_\nu G)(\mu,\mu)| \leq C \mu^{k}.
\end{gathered}\label{inftyorder}
\end{gather}
In particular this shows that, in the sense of distributions, $\wtw$ can be extended across the wedge $\{\mu=0\}\cup \{\nu=0\}$ by setting $\wtw(\mu,\nu,y)= 0$
if $ \mu, \nu \in (-\eps, 0],$  satisfying 
\begin{gather*}
\wtq \wtw=0 \text{ in } (-\eps,\eps)\times (-\eps, \eps) \times \Gamma.
\end{gather*}
We also know more about the regularity of $\wtw$ in the variable $y.$ From \eqref{enerest} we have 
\begin{gather}
\int_{-\eps}^\eps\int_{-\eps}^\eps \int_{\p X} \left[ |\wtw|^2+ \mu|\p_\mu \wtw|^2+ \nu|\p_\nu \wtw|^2+ \mu\nu(\mu+\nu) |d_{h(\mu \nu)} \wtw|^2\right] \, dy d\mu d\nu \leq C ||f||_{L^2(X)}.\label{enerest1}
\end{gather}
The next step is to prove the following unique continuation result
\begin{lemma}\label{unique-cont-2} Let  $\Gamma \subset \p X$ be open and let  $\wtw(\mu,\nu,y) \in C^\infty((-\eps,\eps)\times (-\eps, \eps); L^2( \Gamma))$   satisfy \eqref{enerest1} and be such that for $ y \in \Gamma,$  $\wtw(\mu,\nu,y)$ is supported in   $\{\mu\geq 0, \nu \geq 0\}.$ If 
\begin{gather}
\begin{gathered}
\wtq \wtw(\mu,\nu,y)=0, \text{ in } (-\eps,\eps)\times (-\eps,\eps)\times \Gamma \\
\end{gathered}\label{unique-cont-3}
\end{gather}
then for any  $p \in \Gamma$ there exists $\del>0$ and such that $W(\mu,\nu,y)=0$ provided
$|\mu|<\del,$ $|\nu|<\del$ and $|y-p|<\del.$
\end{lemma}
\begin{proof}  As usual, the proof of this result is based on a Carleman estimate.  First, it is convenient to make the change of variables
\begin{gather*}
r=\frac{\mu+\nu}{2}, \tau=\frac{\mu-\nu}{2},
\end{gather*}
and we have
\begin{gather*}
\wtq_1=4\wtq= \p_r^2-\p_\tau^2 -4(r^2-\tau^2)\Delta_h + (r^2-\tau^2)B_1(r,\tau,y, D_y)+ C_1(r,\tau,y).
\end{gather*}
Here,  $h=h(r^2-\tau^2),$ $B_1(r,\tau,y,D_y)=4B((r^2-\tau^2), y, D_y)$ and $C_1(r,\tau,y)=4C((r^2-\tau^2), y),$ and
\begin{gather*}
\wtw(r,\tau,y)= W(r+\tau, r-\tau, y), \;\ \;\ 
\wtw \text{ is supported in }  |\tau| \leq r.
\end{gather*} 
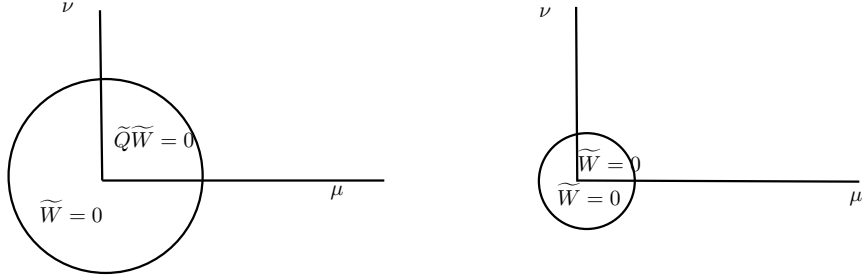
\begin{figure}
% Generated with LaTeXDraw 2.0.8
% Sat Jun 01 02:55:22 EDT 2013
% \usepackage[usenames,dvipsnames]{pstricks}
% \usepackage{epsfig}
% \usepackage{pst-grad} % For gradients
% \usepackage{pst-plot} % For axes
\scalebox{.75} % Change this value to rescale the drawing.
{
\begin{pspicture}(0,-2.44)(15.54,2.48)
\pscircle[linewidth=0.04,dimen=outer](1.74,-0.7){1.74}
\psline[linewidth=0.04cm](10.1,-0.78)(15.1,-0.8)
\usefont{T1}{ptm}{m}{n}
\rput(5.84,-0.995){$\mu$}
\usefont{T1}{ptm}{m}{n}
\rput(1.07,2.285){$\nu$}
\usefont{T1}{ptm}{m}{n}
\rput(15.04,-1.095){$\mu$}
\usefont{T1}{ptm}{m}{n}
\rput(9.53,2.205){$\nu$}
\psline[linewidth=0.04cm](1.64,2.24)(1.68,-0.78)
\psline[linewidth=0.04cm](1.68,-0.78)(6.68,-0.78)
\usefont{T1}{ptm}{m}{n}
\rput(1.13,-1.315){$\wtw=0$}
\usefont{T1}{ptm}{m}{n}
\rput(2.6,-0.035){$\wtq \wtw=0$}
\psline[linewidth=0.04cm](10.1,-0.8)(10.08,2.3)
\usefont{T1}{ptm}{m}{n}
\rput(10.31,-1.015){$\wtw=0$}
\usefont{T1}{ptm}{m}{n}
\rput(10.67,-0.415){$\wtw=0$}
\pscircle[linewidth=0.04,dimen=outer](10.27,-0.79){0.87}
\end{pspicture} 
}
\caption{If $\wtq\wtw=0$ and $\wtw=0$ in the region on the left, then $\wtw=0$ in the region on the right.}
\end{figure}

 Since the fibers over any fixed $(r,\tau)$ are not compact,   one would have to cut-off in the variable $y$ to obtain the desired Carleman estimate.  However this would produce error terms that could not be controlled.  One needs to convexify the support of the solution $\wtw.$  For small $\delta,$ we choose local coordinates $y$ valid in $B(p,\del)$ such that
$p=0,$ and set
\begin{gather*}
\rho= r+|y|^2.
\end{gather*}
In this case the region which contains the support of $W,$ $|\tau|\leq r,$ can be described by 
\begin{gather}
|\tau| + |y|^2 \leq \rho, \label{new-supp}
\end{gather}
and the operator $\wtq_1$ can be written as
\begin{gather*}
\wtq_1= (1+F\vphi) \p_\rho^2-\p_\tau^2+4 \vphi \sum_{ij} h^{ij}\p_{y_i}\p_{y_j} +\sum_j \vphi R_j \p_{y_j}\p_\rho + \sum_j \vphi B_j\p_{y_j} +  H \vphi \p_\rho+ D,
\end{gather*}
where $\vphi=(\rho-|y|^2)^2-\tau^2$ and $F=F(\rho,\tau,y),$ $h^{ij}=h^{ij}(\rho,\tau,y),$ $R_j=R_j(\rho,\tau,y),$ $B_j=B_j(\rho,\tau,y),$  
$D=D(\rho,\tau,y)$ and $H=H(\rho,\tau,y)$ are $C^\infty$ functions.  Here we used that $\Delta_h$ is the positive Laplacian, see \eqref{deftheta}, and hence the sign of the third term. Moreover, there exists a constant $C$ such that for small $\eps,$
\begin{gather*}
\sum_{i,j=1}^n h^{ij}(\rho,\tau,y) \xi_i \xi_j \geq C \sum_j \xi_j^2, \;\  |(\rho,\tau)|<\eps, \;\ y \in \p X.
\end{gather*}

 Let
\begin{gather}
\begin{gathered}
\wtq_{k}=\rho^{-k}\wtq_1 \rho^k=\wtq+k\mcl+ k(k-1)\rho^{-2}(1+\vphi F)- k\rho^{-1} \vphi H , \text{ where }  \mcl \text{ is the vector field }\\
\mcl= 2(1+\vphi F)\rho^{-1}\p_\rho+ \rho^{-1} \vphi \sum_j R_j \p_{y_j}  
\end{gathered}\label{vfieldl} 
\end{gather} 
  In what follows we will denote the inner product
\begin{gather*}
\lan u,v\ran= \int_0^\eps \int_0^\eps \int_{\p X} u(\rho,\tau,y) \overline{v}(\rho,\tau,y) \; dy d\rho d\tau, \\
\text{ and } ||u||^2=\lan u, u \ran,
\end{gather*}
for $\eps$ small enough such that these coordinates are valid.
 We want to estimate the product 
\begin{gather*} 
\lan \wtq_{k} V,  \mcl V\ran \text{ for }
V \in C_0^\infty((-\gamma,\gamma)\times (-\gamma,\gamma)\times \p X), \text{ supported in } |\tau|+|y|^2 \leq \rho.
\end{gather*}
Without loss of generality, we assume that $V$is real valued.     
 From \eqref{new-supp} we find that $\vphi \rho^{-2} \leq C, $ on the support of $V$ and so, for $\gamma$ small enough, again using that $|\tau|\leq \rho-|y|^2,$ there exists a constant $M>0$ such that first term of this product satisfies
\begin{gather}
\begin{gathered}
\lan ((1+\vphi F)\p_\rho^2-\p_\tau^2)V, \mcl V \ran \geq \\
 M (||\rho^{-1} \p_{\rho} V||^2+ ||\rho^{-1} \p_\tau V||^2- \sum_j ||\rho^{-\oq}(\rho-|y|^2)^\ha\p_{y_j} V||^2 ).
 \end{gathered}\label{term1}
 \end{gather}

For $\gamma$ small enough,  the second term can be bounded by
\begin{gather}
\begin{gathered}
\sum_{ij} \lan \vphi h^{ij}\p_{y_i} \p_{y_j} V, \mcl V\ran  \geq M( \sum_j ||\rho^{-\ha}(\rho-|y|^2)^\ha \p_{y_j} V||^2- ||\rho^{-\ha}\p_\rho V||^2).
\end{gathered}\label{term2}
\end{gather}
For $\gamma$ small enough, we also have
\begin{gather}
\sum_{jk} \lan \vphi R_k \p_{y_k}\p_\rho V, \mcl V\ran \geq 
- M( \sum_j ||(\rho-|y|^2)\p_{y_j} V||^2- ||\p_\rho V||^2-||V||^2), \label{term2new}
\end{gather}
and 
\begin{gather}
\begin{gathered}
\lan (\vphi \sum_j B_j \p_{y_j}+ D+ \vphi H \p_\rho) V,  \mcl V \ran \geq  - M( ||V||^2+ ||\rho^{-1} \p_\rho V||^2 +\sum_j ||(\rho-|y|^2)^\ha \p_{y_j} V||^2). \end{gathered}\label{term3}
\end{gather}
Finally,
\begin{gather}
\lan \left( k(k-1)\rho^{-2}(1+\vphi F)+ k\rho^{-1} \vphi H\right) V, \mcl V \ran\geq  M k^2 ||\rho^{-2} V||^2. \label{term4}
\end{gather}
Putting together terms \eqref{term1}, \eqref{term2}, \eqref{term2new} \eqref{term3} and \eqref{term4} we deduce that, for $\gamma$ small enough, there exists $M>0$
\begin{gather*}
\lan \wtq_{k} V,  \mcl V\ran \geq   M( ||\rho^{-1} \p_\rho V||^2 +  || \rho^{-1} \p_\tau V||^2 + 
\sum_j ||\rho^{-\ha}(\rho-|y|^2) \p_{y_j} V||^2 + k ||\mcl V||^2 + k^2 ||\rho^{-2} V||^2).
\end{gather*}
Since,
\begin{gather*}
\lan \wtq_{k} V, \mcl V \ran \leq \ha( ||\wtq_k V||^2+ ||\mcl V||^2)
\end{gather*}
then,  if $k$ is large enough and $\gamma$ is small enough,
\begin{gather}
 ||\wtq_{k} V||^2 \geq   M( ||\rho^{-1} \p_\rho V||^2+ || \rho^{-1} \p_\tau V||^2 + 
 ||\rho^{-\ha}(\rho-|y|^2)^\ha \nabla_y V||^2 + \frac{k}{2} ||\mcl V||^2 + k^2 ||\rho^{-2} V||^2).\label{carle}
\end{gather}

Let $\chi\in C^\infty(\mr),$ $\chi(\rho)=1$ if $\rho\in (-\frac{\gamma}{2},\frac{\gamma}{2})$ and $\chi(\rho)=0$ if $|\rho|>\frac{3\gamma}{4}.$ Since $W$ is supported in $\mu\geq 0,$ $\nu\geq 0,$ it follows that
in coordinates $(\rho,\tau,y),$ $\wtw$ is supported in $|\tau|+ |y|^2 \leq \rho,$ and hence $V=\chi(\rho) \wtw$ is compactly supported for small $\gamma.$  We would like to apply \eqref{carle} to $V=\chi(\rho) \wtw,$ but $\wtw$ is not necessarily  smooth up to $\{\mu=0\},$ $\{\nu=0\}.$  So we have to molify $\wtw$ in the $y$-variable, and we  let
$\psi(y)$ be a $C_0^\infty$ function supported in $|y-p|<\del,$ with $\int \psi(y) \; dy=1,$ and define
$\psi_m(y)=m^{n}\psi(m y).$ Then for $m$ large, 
\begin{gather*}
\wtw_m=\chi(\rho)\psi_m*\wtw \in C_0^\infty((-\eps,\eps)\times (-\eps,\eps) \times \Gamma),
\end{gather*}
and since $\wtq_k= \rho^{-k}\wtq_1 \rho^k,$  we deduce from \eqref{carle} that
\begin{gather}
\begin{gathered}
  ||\rho^{-k} \wtq \wtw_m||^2 \geq   \\ M\left( k^2 || \rho^{-2-k}  \wtw_m||^2+  ||\rho^{-\ha-k}(\rho-|y|^2)^\ha \nabla_y \wtw_m||^2 +  ||\rho^{-1}\p_\rho(\rho^{-k} \wtw_m)||^2+
   ||\rho^{-1-k}\p_\tau \wtw_m)||^2 \right).
\end{gathered}  \label{carle1}
\end{gather}
To get an estimate for $\chi(\rho)\wtw$ from this one we use  Friedrich's lemma to handle the commutators of $\wtq$ and $\psi_m.$ We use \eqref{enerest1} and Theorem 2.4.3 of \cite{hormander64} to show that
\begin{gather*}
||\vphi \chi(\rho)\left( h^{ij}\p_{y_i}\p_{y_j} (\psi_m * \wtw)-(h^{ij}\p_{y_i}\p_{y_j} \wtw)*\psi_m\right)||_{L^2} \leq C|| \chi(\rho)\vphi \nabla_y \wtw||_{L^2}, \\
||\vphi h^{ij}\p_{y_i} (\psi_m *\p_\rho(\rho^{-k} \chi(\rho)\wtw))-\vphi (h^{ij}\p_{y_i}\p_{\rho}( \rho^{-k}\chi(\rho)\wtw))*\psi_m||_{L^2} \leq C||\vphi \p_\rho \rho^{-k}\chi(\rho)\wtw||_{L^2}, \\
||\rho^{-k}\chi(\rho)(\rho-|y|^2)^\ha(\nabla_y (\psi_m * \wtw)-(\nabla_y \wtw)*\psi_m)||_{L^2} \leq C|| \rho^{-k}\chi(\rho)(\rho-|y|^2)^\ha \wtw||_{L^2}, 
\end{gather*}
and moreover
\begin{gather*}
\lim_{m\rightarrow \infty} ||\vphi \chi(\rho)\left( h^{ij}\p_{y_i}\p_{y_j} (\psi_m * \wtw)-(h^{ij}\p_{y_i}\p_{y_j} \wtw)*\psi_m\right)||_{L^2} =0 \\
\lim_{m\rightarrow \infty} ||\vphi h^{ij}\p_{y_i} (\psi_m *\p_\rho\rho^{-k} \chi(\rho)\wtw)-\vphi (h^{ij}\p_{y_i}\p_{\rho} \rho^{-k}\chi(\rho)\wtw)*\psi_m||_{L^2}=0 \\
\lim_{m\rightarrow \infty} ||\rho^{-k}\chi(\rho)(D (\psi_m * \wtw)-(D \wtw)*\psi_m)||_{L^2}=0.
\end{gather*}
Using these estimates, and letting $m\rightarrow \infty$ in \eqref{carle1} we obtain, for $k$ large enough,
\begin{gather}
\begin{gathered}
 ||\rho^{-k} \wtq_1 \chi(\rho) \wtw||^2 \geq   \\ M\left( k^2 || \rho^{-2-k} \chi(\rho) \wtw||^2+ ||\rho^{-k-\ha} (\rho-|y|^2) \chi(\rho) \nabla_y \wtw||^2+||\rho^{-1}\p_\rho(\rho^{-k} \wtw)||^2+ ||\rho^{-1-k} \p_\tau (\chi(\rho) \wtw||^2\right).
\end{gathered} \label{carle2}
\end{gather}

Since $\wtq_1 \wtw=0$ and $\wtq_1 (\chi(\rho) \wtw)=[\wtq_1, \chi(\rho)] \wtw$ is supported in $\rho\geq \frac{\gamma}{2},$ we deduce from \eqref{carle2} that there exists $C=C(\wtw)>0$ such that

\begin{gather*}
C (\frac{\gamma}{2})^{-k} \geq M k^2 ||\rho^{-2-k} \wtw||^2\geq M k^2 ||\rho^{-2-k} \wtw||_{L^2(\rho\leq \frac{\gamma}{2})}^2\geq M k^2 (\frac{\gamma}{2})^{-k-2} 
|| \wtw||_{L^2(\rho\leq \frac{\gamma}{2})}^2.
\end{gather*}
Hence
\begin{gather*}
k^2 || \wtw||_{L^2(\rho\leq \frac{\gamma}{2})}^2 \leq C,
\end{gather*}
and therefore $\wtw=0$ if $\rho \leq \frac{\gamma}{2}.$  Returning to coordinates $(r,\tau,y),$ we obtain
$\wtw=0$ if $r+|y|^2 \leq \frac{\gamma}{2}$ and so $\wtw(r,\tau,y)=0$ i $r<\frac{\gamma}{4}$ and $|y|\leq \frac{\gamma}{4}.$ This ends the proof of the Lemma.
\end{proof}

Since $W(\mu,\nu,y)= V_+(\mu\nu, 2\log \nu,y)$ satisfies the hypotheses of Lemma \ref{unique-cont-2}, we conclude that for any $p\in \Gamma,$ there exists  $\del>0$ and such that
\begin{gather*}
V_+(\mu\nu, 2\log \nu,y)=0 \text{ provided } |y-p|<\del, \text{ and }  \mu, \nu \in (-\del, \del).
\end{gather*}
In view of \eqref{defmunu}, we deduce that
\begin{gather}
x^{-\frac{n}{2}} u(s-\log x, x, y)=V_+(x,s,y)=0 \text{ provided }  |y-p|<\del,  \; x \in (0, \del) \text{ and } \log x < s< \log \del. \label{region-van}
\end{gather}
We have also shown that $V_+$ can be extended to the region $x<0$ such that, for $P$ as in \eqref{operatorP}
\begin{gather}
\begin{gathered}
PV_+=0 \\
V_+(x,s,y)=0 \;\ x<0, \;\ s< s_0 \text{ and } y\in \Gamma. 
\end{gathered}\label{vanish-v}
\end{gather}
\end{proof}

The next step in the proof is
\begin{prop}\label{step-2}  Let  $V(x,s,y)$ be in $H^1_\loc$ in the region $|x|<\eps,$ $y \in \p X$ and $s \in \mr,$ satisfy $PV=0,$ where $P$ is given by \eqref{operatorP}.  Suppose 
$V(x,s,y)=0$ for $x\in (-\eps,0),$ $s\leq s_0$ and $y \in \Gamma.$  Let $s_1<s_0$  and 
$p \in \p X,$ and suppose that  there exists $\del>0$   such that $V(x,s,y)=0$ if $x< \del,$   $|y-p|<\del$ and $s< s_1.$ 
Then there exists $\beta\in (0,\del)$ such that $V(x,s,y)=0$ if $x<\beta,$ $|y-p|<\beta$ and $s<s_1+\frac{1}{4}(s_0-s_1).$ Figure \ref{fig3} illustrates the result. 
\end{prop}
\begin{figure}[htb!]
% Generated with LaTeXDraw 2.0.8
% Wed Jun 05 17:27:16 EDT 2013
% \usepackage[usenames,dvipsnames]{pstricks}
% \usepackage{epsfig}
% \usepackage{pst-grad} % For gradients
% \usepackage{pst-plot} % For axes
\scalebox{.75} % Change this value to rescale the drawing.
{
\begin{pspicture}(0,-3.97)(18.96,3.99)
\psline[linewidth=0.04cm](0.8,3.75)(0.8,-3.69)
\psline[linewidth=0.04cm](0.82,0.43)(7.26,0.47)
\psbezier[linewidth=0.04](2.38,-3.69)(2.4,0.21)(3.7461386,2.124097)(6.5,3.0538332)
\psline[linewidth=0.04cm](0.8,-0.73)(2.76,-0.75)
\usefont{T1}{ptm}{m}{n}
\rput(1.76,-1.885){$PV=0$}
\psline[linewidth=0.04cm](11.98,3.39)(12.04,-3.95)
\psline[linewidth=0.04cm](12.02,0.43)(18.22,0.45)
\psbezier[linewidth=0.04](13.5,-3.85)(13.520874,0.05)(14.925824,1.9640971)(17.8,2.8938332)
\psline[linewidth=0.04cm](12.0,-0.71)(13.96,-0.73)
\usefont{T1}{ptm}{m}{n}
\rput(7.39,0.435){$x$}
\usefont{T1}{ptm}{m}{n}
\rput(1.19,3.795){$s$}
\usefont{T1}{ptm}{m}{n}
\rput(18.65,0.455){$x$}
\usefont{T1}{ptm}{m}{n}
\rput(12.65,3.515){$s$}
\usefont{T1}{ptm}{m}{n}
\rput(6.65,2.635){$s=\log x$}
\usefont{T1}{ptm}{m}{n}
\rput(17.39,2.155){$s=\log x$}
\psline[linewidth=0.04cm,linestyle=dashed,dash=0.16cm 0.16cm](2.42,-2.51)(2.42,0.47)
\usefont{T1}{ptm}{m}{n}
\rput(2.43,0.755){$e^{s_1}$}
\usefont{T1}{ptm}{m}{n}
\rput(0.51,-0.325){$s_0$}
\usefont{T1}{ptm}{m}{n}
\rput(11.39,-0.405){$s_0$}
\psline[linewidth=0.04cm,linestyle=dashed,dash=0.16cm 0.16cm](13.54,-2.73)(13.52,0.49)
\usefont{T1}{ptm}{m}{n}
\rput(13.81,0.775){$e^{s_1}$}
\psline[linewidth=0.04,fillstyle=gradient,gradlines=2000,gradmidpoint=1.0](0.02,-0.73)(0.8,-0.73)(0.78,-3.75)(0.0,-3.75)(0.02,-3.77)
\psline[linewidth=0.04,fillstyle=gradient,gradlines=2000,gradmidpoint=1.0](11.22,-0.71)(12.04,-0.71)(12.02,-3.95)(11.2,-3.93)(11.22,-3.95)
\psline[linewidth=0.04,fillstyle=gradient,gradlines=2000,gradmidpoint=1.0](0.82,-2.73)(1.62,-2.73)(1.62,-3.73)(0.8,-3.73)
\psline[linewidth=0.04,fillstyle=gradient,gradlines=2000,gradmidpoint=1.0](12.04,-3.95)(12.84,-3.95)(12.82,-2.71)(12.02,-2.71)(12.04,-2.73)
\psline[linewidth=0.04cm,linestyle=dashed,dash=0.16cm 0.16cm](1.64,-2.75)(2.44,-2.73)
\psline[linewidth=0.04cm,linestyle=dashed,dash=0.16cm 0.16cm](12.86,-2.73)(13.58,-2.73)
\psline[linewidth=0.04cm,linestyle=dashed,dash=0.16cm 0.16cm](12.38,0.41)(12.42,-1.93)
\psline[linewidth=0.04,fillstyle=gradient,gradlines=2000,gradmidpoint=1.0](12.04,-1.95)(12.06,-1.95)(12.06,-1.95)(12.44,-1.95)(12.42,-2.71)(12.04,-2.69)(12.06,-2.71)
\usefont{T1}{ptm}{m}{n}
\rput(12.51,0.715){$\beta$}
\end{pspicture} 
}
\caption{The unique continuation across the wedge $\{s<s_1,\;  x<\del,\;  |y-p|<\del\} \cup \{ x<0, \; s<s_0,\;  |y-p|<\del\}.$ If $PV=0$ and $V=0$ in the colored region on the left, then $V=0$ in the colored region on the right.}
\label{fig3}
\end{figure}
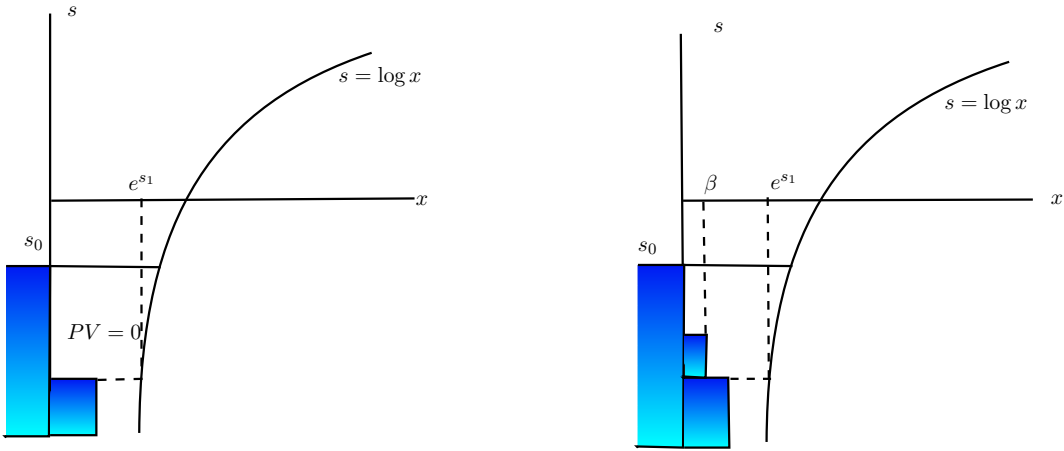
\begin{proof}  
 We will use unique continuation results due to H\"ormander, and we will need to work with suitable strictly pseudoconvex surfaces.  The key point here is that the operator $P$ is invariant under translations in the variable $s.$  Let
\begin{gather*}
\vphi(x,s,y)=-x-k(s-s_1)-|y-p|^2, \text{ where } k> 0 \text{ will be chosen later}.
\end{gather*}
Since for $|y-p|<\del,$ $V=0$ if $x\in (-\eps,0]$ and  $s<s_0,$ or if $x<\del$ and $x<s_1,$  we have, see figure \ref{fig3},
\begin{gather}
V(x,s,y)=0 \text{ if } \vphi>0, \;\ -\eps< x <\del, \;\  \text{ and } |y-p|<\del. \label{vanish-v1}
\end{gather}

 The principal symbol of the operator $P$ is
\begin{gather*}
p= -2\sigma \xi -x \xi^2-x h(x,y,\eta),
\end{gather*}
where $(\xi,\sigma,\eta)$ are the dual variables to $(x,s,y).$ Since $\nabla\vphi(x,s,y)=(-1,-k,-2(y-p)),$ we have
\begin{gather}
p(x,s,y,\nabla\vphi(x,s, y))=-2k-x-h(x,y, 2(y-p)), \label{nonch}
\end{gather}
 $\vphi$ is not characteristic at $(x,s,y)$ if $x>-\frac{2k}{1+h(x,y,-2(y-p))}>-k$ if $|y-p|<\beta$ is small enough  

The Hamiltonian of $p$ is
\begin{gather*}
H_p= -2\xi \p_s -2(\sigma+x \xi) \p_x -xH_{h} + (\xi^2+h+ x\p_x h) \p_\xi
\end{gather*}
where $H_h$ denotes the Hamiltonian of $h(x,y,\eta)$ in the variables $(y,\eta).$  Hence,
\begin{gather*}
(H_p\vphi)(x,s,y,\xi,\sigma,\eta) = 2(\sigma+x\xi)+2k \xi+x H_h |y-p|^2 \text{ and } \\
(H_p^2\vphi)(x,s,y,\xi,\sigma,\eta) = \\ -2(\sigma+x\xi)(2\xi +H_h |y-p|^2+ x\p_x H_h |y-p|^2) - (x H_h)^2 |y-p|^2+ 2(k+x)(\xi^2+h+x\p_x h).
\end{gather*}
If $H_p \vphi=0$, it follows that 
\begin{gather*}
H_p^2\vphi(x,s,y,\xi,\sigma,\eta)=
2(x+3k)\xi^2+ 2\xi( (k+x) H_h|y-p|^2+x \p_x H_h|y-p|^2) +  \\ 2(k+x)(h+x\p_x h)+ x\left( (H_h|y-p|^2)^2+ x H_h|y-p|^2 \p_x H_h|y-p|^2-x H_h^2 |y-p|^2\right).
\end{gather*}
If $|y-p|<\beta$ is small enough, there exists $C>0$ depending on $h$ only such that
\begin{gather*}
\left|H_p |y-p|^2\right| \leq C \beta |\eta|, \text{ and }
\left|\p_x H_p |y-p|^2\right| \leq C \beta |\eta|.
\end{gather*}
If we impose that  $-\frac{k}{2}<x<\beta,$ it follows that  there exists $\eps_0>0$ depending on $h$ such that if $\beta, k\in (0,\eps_0)$ small,  there exists $C>0$ such that 
\begin{gather*}
h+ x\p_x h \geq C |\eta|^2,
\end{gather*}
and hence
 \begin{gather*}
 H_p^2\vphi(x,s,p,\xi,\sigma,\eta)\geq kC (\xi^2- \beta |\xi||\eta|^2+ |\eta|^2) \geq C k (\xi^2+ |\eta|^2),  \\ \text{ if } -\frac{k}{2}<x <\beta \text{ and } k, \del \in (0,\eps_0).
 \end{gather*}

 So we conclude that there exists $\eps_0>0$ depending on $h$ such that
 \begin{gather}
  \begin{gathered}
 \text{ if } p(x,s,y,\xi,\sigma,\eta)=H_p\vphi(x,s,y,\xi,\sigma,\eta)=0 \text{ then } H_p^2 \vphi(x,s,y,\xi,\sigma,\eta)>0 \\
 \text{ provided } (\xi,\sigma,\eta)\not=0, \;\ -\frac{k}{2}<x<\beta, \;\ |y-p|<\beta, \;\ k, \beta \in (0,\eps_0).
 \end{gathered}\label{strpc}
  \end{gather}

 Since $P$ is of second order, we deduce from \eqref{nonch} and \eqref{strpc} that  the level surfaces of $\vphi$ are  strictly pseudoconvex in the region
 \begin{gather}
 -\frac{k}{2}<x <\beta, \;\ |y-p|<\beta, \text{ provided } k, \beta \in (0,\eps_0). \label{strpc-region}
 \end{gather}
 
As mentioned above, the invariance of $P$ under translations in $s$ imply that the conditions in \eqref{strpc-region}  do not depend on $s.$ Now we appeal to Theorem 28.2.3 and Proposition 28.3.3  of  \cite{Ho}  and conclude that if 
 \begin{gather*}
 Y=\{ -\frac{k}{4}<x<\frac{\beta}{2}, \;\  |y-p|<\frac{\beta}{\sqrt{2}},\;\  |s-s_1| < s_0-s_1\},
 \end{gather*}
  there exist
 $C>0$ and $\la>0$ large such that if $\psi=e^{\la \vphi},$
 \begin{gather}
C ||e^{\tau \psi} P v||^2\geq \tau^2 ||e^{\tau \psi} v||^2 +  \tau ||e^{\tau \psi} v||_{H^1}^2, \text{ for all } v \in C_0^\infty(Y) \text{ and } \tau\geq \tau_0>0. \label{carle-hor}
\end{gather}
 Let $\theta \in C_0^\infty(Y)$ with $\theta=1$ if $-\frac{k}{8}<x< \frac{\beta}{4},$ $|y-p|<\frac{\beta}{2}$ and $|s-s_1|< \alpha (s_0-s_1),$ $\alpha<1.$
Since $PV=0,$ it follows that  
\begin{gather*}
P(\theta V)=[P,\theta] V.
\end{gather*}
But for $(x,s,y)\in Y,$ $V(x,s,y)$ is supported in the region $x>0,$ $s>s_1,$  so we conclude that
\begin{gather*}
P(\theta(x,s,y) V) \text{ is supported in  } (x,s,y) \in Y \;\  x\geq \frac{\beta}{4}, \;\ s-s_1\geq \alpha (s_0-s_1), \;\ \alpha<1  \text{ and } |y-p|\geq \frac{\beta}{2}.
\end{gather*}
Therefore, by the definition of $\vphi$ we have
\begin{gather}
\vphi(x,s,y)\leq -\min\{ \frac{\beta}{4}, k\alpha (s_0-s_1), \frac{\beta^2}{4}\}  
\text{ on the support of } P(\theta V). \label{minphi}
\end{gather}
 Pick $k$ small so that $\min\{ \frac{\beta}{4}, k\alpha (s_0-s_1), \frac{\beta^2}{4}\}  = k\alpha(s_0-s_1)=\gamma.$ 
 Therefore we deduce from \eqref{carle-hor} and \eqref{minphi} that
\begin{gather*}
\tau^2 ||e^{\tau (e^{\la \vphi}- e^{-\la\gamma})}\theta V||^2 \leq C, \;\ \tau>\tau_0.
\end{gather*}
We remark that due to Friedrichs' Lemma,  one can apply \eqref{carle-hor} to $\theta V$ even though it is not $C^\infty,$ see \cite{Ho}. Therefore,
 $\theta V=0$ if $e^{\la \vphi}-e^{-\la \gamma}>0,$ so $\theta V=0$ if $\vphi >-\gamma.$
 So we deduce that 
 \begin{gather*}
 \theta V(x,s,y)=0 \text{ provided }  k(s-s_1)<  \frac{\gamma}{3}, \;\ 0< x < \frac{\gamma}{3} \;\ |y-p|^2<\frac{\gamma}{3}. 
 \end{gather*}
In particular,
 \begin{gather}
 V(x,s,y)=0 \text{ provided } s <s_1+\frac{\alpha}{3}(s_0-s_1), \;\ \alpha<1, \;\  0 < x < \frac{\gamma}{3}, \;\ |y-p|^2<\frac{\gamma}{3}. \label{vanish-v-2}
 \end{gather}
 This concludes the proof of Proposition \ref{step-2}.  
\end{proof}
 
 The final ingredient in the proof of Theorem \ref{support-local} is
 \begin{prop}\label{step-3}  Let $u(t,z)$ satisfy \eqref{wave} with initial data $f_1=0,$ $f_2=f\in L^2_{ac}(X)\cap C^\infty(\intx).$  Let $V_+(x,s,y)=x^{-\novt} u(s-\log x, x,y).$  Let  $p \in \p X,$ and suppose that  there exist $s_2\in \mr,$ $\gamma>0$ and $\del>0$   such that $V_+(x,s,y)=0$ if $0<x<\gamma,$ $\log x<  s<s_2$ and $|y-p|<\del.$  
Then  $u(t,z)=0$ for every $z\in X$ such that  there  exist  $(x,y)$ with $x<\gamma$ and $|y-p|<\del$ and $|t|+ d_g(z; (x,y))\leq \log (\frac{e^{s_0}}{x}).$ In particular, if $s^*<s_2$ is such that coordinates $(x,y)$ for which \eqref{metric} holds for $x <e^{s^*},$ then  
\begin{gather}
\begin{gathered}
V_+(x,s,y)=0 \text{ if }  |y-p|<\del, \;\   0<x < e^{s^*}, \text{ and }  \log x< s < s_2.
\end{gathered}\label{setv+}
\end{gather}
 Figure \ref{fig4} illustrates the result.
 \end{prop}
 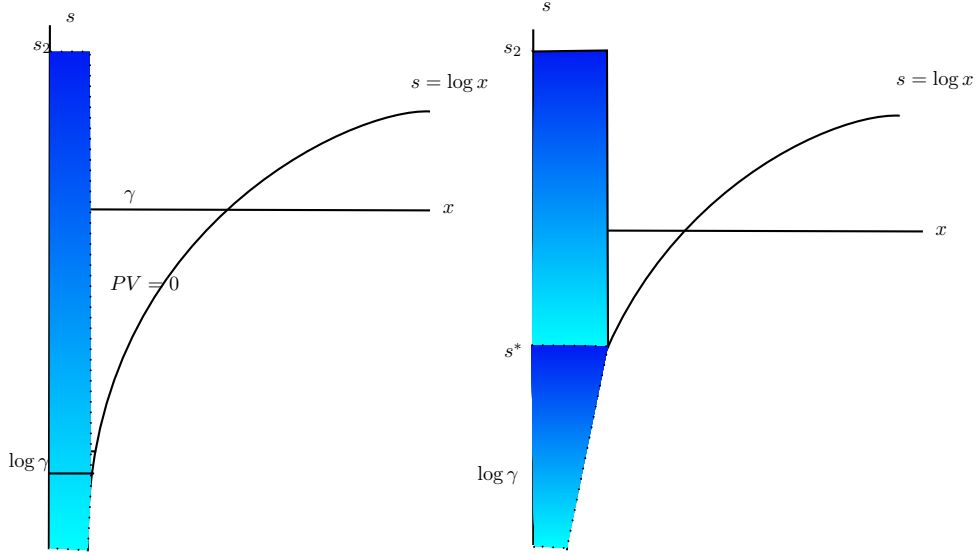
\begin{figure}[htb!]
% Generated with LaTeXDraw 2.0.8
% Thu Jul 04 08:54:47 EDT 2013
% \usepackage[usenames,dvipsnames]{pstricks}
% \usepackage{epsfig}
% \usepackage{pst-grad} % For gradients
% \usepackage{pst-plot} % For axes
\scalebox{.7} % Change this value to rescale the drawing.
{
\begin{pspicture}(0,-5.27)(19.1,5.29)
\psline[linewidth=0.04cm](1.88,1.25)(8.7,1.23)
\usefont{T1}{ptm}{m}{n}
\rput(9.05,1.255){$x$}
\usefont{T1}{ptm}{m}{n}
\rput(1.87,4.895){$s$}
\psbezier[linewidth=0.04](2.2,-5.25)(2.08,0.83)(7.22,3.19)(8.7,3.11)
\usefont{T1}{ptm}{m}{n}
\rput(3.28,-0.145){$PV=0$}
\psline[linewidth=0.04cm](18.06,0.83)(10.66,0.85)
\usefont{T1}{ptm}{m}{n}
\rput(18.41,0.835){$x$}
\usefont{T1}{ptm}{m}{n}
\rput(10.91,5.095){$s$}
\usefont{T1}{ptm}{m}{n}
\rput(9.07,3.655){$s=\log x$}
\usefont{T1}{ptm}{m}{n}
\rput(18.29,3.715){$s=\log x$}
\psbezier[linewidth=0.04](11.3,-5.23)(11.3,0.8534426)(16.14,3.11)(17.62,3.0302186)
\usefont{T1}{ptm}{m}{n}
\rput(1.33,4.335){$s_2$}
\usefont{T1}{ptm}{m}{n}
\rput(10.27,4.275){$s_2$}
\usefont{T1}{ptm}{m}{n}
\rput(3.0,1.495){$\gamma$}
\usefont{T1}{ptm}{m}{n}
\rput(1.09,-3.545){$\log \gamma$}
\usefont{T1}{ptm}{m}{n}
\rput(9.99,-3.805){$\log \gamma$}
\usefont{T1}{ptm}{m}{n}
\rput(10.28,-1.425){$s^*$}
\psline[linewidth=0.04cm](1.48,4.75)(1.46,-5.21)
\psline[linewidth=0.04cm](10.66,4.67)(10.66,-5.13)
\psline[linewidth=0.04,fillstyle=gradient,gradlines=2000,gradmidpoint=1.0](10.64,-1.41)(10.7,-1.37)(12.08,-1.37)(12.06,4.27)(10.66,4.25)(10.68,4.23)
\psline[linewidth=0.04cm](1.46,-3.35)(2.36,-3.35)
\psline[linewidth=0.04,linestyle=dotted,dotsep=0.16cm,fillstyle=gradient,gradlines=2000,gradmidpoint=1.0](1.48,4.25)(2.24,4.25)(2.26,-3.95)(2.2,-5.23)(1.46,-5.19)(1.48,-5.21)
\psline[linewidth=0.04cm](1.46,-3.77)(2.32,-3.77)
\psline[linewidth=0.04cm](10.68,-3.73)(11.44,-3.73)
\psline[linewidth=0.04,linestyle=dotted,dotsep=0.16cm,fillstyle=gradient,gradlines=2000,gradmidpoint=1.0](10.62,-1.33)(12.08,-1.35)(11.3,-5.19)(10.66,-5.15)(10.68,-5.17)
\end{pspicture} 
}
\caption{ If  $PV=0$ and $V=0$ in the colored region on the left, then $V=0$ in the colored region on the right.}
\label{fig4}
\end{figure}
\begin{proof}  The key point in the proof is the following consequence of  Tataru's theorem \cite{TAT}, see also \cite{tataru1,hormander1,robbianozuilly}.  If $u(t,z)$ is a $C^\infty$ function that satisfies
\begin{gather*}
\begin{gathered}
(D_t^2-\Delta_g +L(z,D_z))u=0 \text{ in } (\widetilde{T}, \widetilde{T}) \times \Omega, \\
u(t,z)=0 \text{ in a neighborhood of } \{z_0\} \times (-T,T), \;\ T<\widetilde{T},
\end{gathered}
\end{gather*}
where $\Omega\subset \mrn,$ $g$ is a $C^\infty$ Riemannian metric and $L$ is a first order $C^\infty$ operator (that does not depend on $t$), then
\begin{gather}
u(t,z)=0 \text{ if } |t|+ d_g(z,z_0)<T, \label{tatres}
\end{gather}
 where $d_g$ is the distance measured with respect to the metric $g.$  

Since the initial data of \eqref{wave} is $(0,f),$  $u(t,z)=-u(-t,z).$ If  $x<\gamma,$ $\log x<s<s_1,$ and $|y-p|<\del,$ 
 it follows from the definition of $V_+$ that
\begin{gather*}
u(t,x,y)=0 \text{ if }  0<x<\gamma, \;\  |y-p|<\del \text{ and } |t| \leq s_2-\log x= \log (\frac{e^{s_2}}{x}).
\end{gather*}

Applying \eqref{tatres} with $z_0=(x,y)$ we obtain
\begin{gather*}
u(t,z)=0 \text{ provided } |t|+d_g(z; (x,y))< \log (\frac{e^{s_2}}{x}), \;\ \text{ with } 0<x<\del, \; |y-p|<\del.
\end{gather*}

If $z=(\alpha,y)$ with $e^{s^*}>\alpha>x,$ $d_g((x,y); (\alpha,y))=  \log(\frac{\alpha}{x}),$  it follows from \eqref{tatres}
 \begin{gather*}
u(t,(\alpha,y))=0 \text{ if } t+\log(\frac{\alpha}{x})< \log(\frac{e^{s_2}}{x}). \label{applic-tat}
\end{gather*}
In particular this guarantees that
$u(t,\alpha,y)=0$ if $ 0<t< \log(\frac{e^{s_2}}{ \alpha}),$  and  since $s=t+\log \alpha,$ hence $V_+(\alpha,s,y)=0$ if $\alpha<e^{s^*},$ $s<s_2$ and $|y-p|<\del.$ This ends the proof of Proposition \ref{step-3}.
\end{proof}
We can now conclude the proof of Theorem \ref{support-local}. 
\begin{proof}
 We know from Proposition \ref{step-1} that  for any $p\in \Gamma$ there exists $\del>0$ such that 
\begin{gather*}
V_+(x,s,y)=0 \text{ if }  x< \del , \;\  s < \log \del, \;\ |y-p|<\del.
\end{gather*}
Moreover, $V(x,s,y)=0$ if $x<0,$ $s<s_0$ and $y\in \Gamma.$  Applying Proposition \ref{step-2} with $s_1=\log \del,$ we find  that  there exists $\beta_1<\del$ such that
\begin{gather*}
V_+(x,s,y)=0 \text{ provided }  x<\beta_1, \;\ |y-p|<\beta_1 \text{ and } \log x< s< \log \del+ \frac{1}{4}(s_0-\log \del).
\end{gather*}
Then Proposition \ref{step-3} guarantees that there exists $s^*<<0$ such that,
\begin{gather*}
V_+(x,s,y)=0 \text{ if } x<e^{s^*}, \;\ |y-p|<\beta_1, \;\ s<s_2=\log \del+ \frac{1}{4}(s_0-\log \del).
\end{gather*}

Again using  Proposition \ref{step-2}  and Proposition \ref{step-3} $n$ times, we find that there exists $\beta_n<\beta_{n-1}$ such that for 
$s_n=s_{n-1}+\frac{1}{3}(s_0-s_{n-1}),$ 
\begin{gather*}
V_+(x,s,y)=0 \text{ if } x<e^{s^*}, |y-p|<\beta_{n-1}, \;\ s_n= s_{n-1}+\oq(s_0-s_{n-1}).
\end{gather*}
The main point is that while the neighborhood of $p$ shrinks from one step to the next, the neighborhood of $x=0$ stays the same: $x<e^{s^*}.$
Since $p$ is arbitrary, it follows that, for every $p \in \Gamma,$ $V_+(x,s,p)=0$ provided $0< x< e^{s^*}$ and  $\log x<s <s_0.$  

We again resort to the consequence of Tataru's theorem to finish the proof of the result.  Let $w=(\alpha,p),$ with $0<\alpha<e^{s^*}$ and $p\in \Gamma,$ then the solution $u(t,z)$ vanishes in a neighborhood of $\{w\} \times \left(-\log(\frac{e^{s_0}}{\alpha}), \log(\frac{e^{s_0}}{\alpha})\right).$ Therefore, by  \eqref{tatres},
\begin{gather*}
u(t,z)=\p_tu(t,z)=0  \text{ if } |t|+ d_g(z,w)< \log(\frac{e^{s_0}}{\alpha}).
\end{gather*}
In particular, when $t=0$ we find that   $\p_tu(0,z)=f(z)=0$ provided $d_g(z,w)< \log(\frac{e^{s_0}}{\alpha}),$  and this concludes the proof of Theorem \ref{support-local}.
\end{proof}

The following result will be useful in the next section.
\begin{cor}\label{uniqueinf}   Let $\Gamma\subset \p X$ be open.  If $f\in L^2_{ac}(X)$ and $\mcr_+(0,f)(s,y)=0$ for every $s\in \mr$ and $y\in \Gamma,$ then $f=0.$ Similarly, if 
$(h,0)\in E_{ac}(X)$ and $\mcr_+(h,0)(s,y)=0$ for every $s\in \mr$ and $y\in \Gamma,$ then $h=0.$
\end{cor}
\begin{proof}  If $\mcr(0,f)(s,y)=0$ for every $s\in \mr$ and $y\in \Gamma,$ then $f(z)=0$ if $z\in \mcd_{s_0}(\Gamma)$ for every $s_0.$  Since $X$ is connected, the distance between any two points in the interior of $X$ is finite.  Therefore $f=0.$  

Suppose $F=\mcr(h,0)(s,y)=0$ for every $s\in \mr$ and $y\in \Gamma.$   As in the proof of Theorem \ref{support-local}, by taking convolution of $F$ with  
$\phi \in C_0^\infty(\mr),$ we may assume that $(\Delta_g-\frac{n^2}{4})^k h\in L^2_{ac}(X)$ for every $k\geq 0.$  
Let $u(t,z)$ satisfy \eqref{wave} with initial data $(h,0)$ and let $V=\p_t u.$  Then $V$ satisfies \eqref{wave} with initial data $(0, (\Delta_g-\frac{n^2}{4})h)$ and
$\mcr_+\left( 0,(\Delta_g-\frac{n^2}{4})h\right)(s,y)=0$ for $s\in \mr$ and $y\in \Gamma.$ But as we have shown, this implies that
$(\Delta_g-\frac{n^2}{4})h=0.$ Since $(h,0)\in E_{ac}(X),$ this implies that $h=0.$
\end{proof}

One should remark that this result can be proved by applying a result of Mazzeo \cite{mazzeo2}, see also \cite{VasyWunsch}.  The solution to \eqref{wave} with initial data $(0,f)$ is odd,and since
 $\mcr_+(0,f)(s,y)=0,$ for $s\in \mr,$$y\in \Gamma,$  it follows that $\mcr_-(0,f)(s,y)=0$ for $s\in \mr,$ $y \in \Gamma.$   Taking Fourier transform in $s$ we find that 
 \begin{gather*}
 (\Delta_g-\la^2-\frac{n^2}{4}) \widehat{u}(\la,z)=0
 \end{gather*}
and using that $\mcr_+(0,f)(s,y)=\mcr_-(0,f)(s,y)=0,$ one deduces by using a formal power series argument as in the proof of Proposition 3.4 of \cite{GraZwo},  that $\widehat{u}(\la,z)$ vanishes to infinite order at $\Gamma.$ Theorem 14 of \cite{mazzeo2} implies that $\widehat{u}=0$ and hence $u=0.$ In particular $f=0.$

\section{The Control Space}

One of the key arguments used in \cite{sa}  to prove Theorem \ref{inverse-full} was that the ranges of the forward and backward radiation fields
\begin{gather*}
\mcm^{\pm}= \mcr_{\pm}(0,L^2_{ac}(X))= \{\mcr_{\pm}(0,f): f \in L^2_{ac}(X)\}
\end{gather*}
are closed subspaces of $L^2(\mr \times \p X),$ and are characterized by the scattering operator.  
  The main goal of this section is to define an analogue of $\mcm^{\pm}$ for  functions supported in $\mr \times \Gamma.$ Throughout the remaining of this paper we will denote
\begin{gather*}
L^2(\mr \times \Gamma)=\{ F\restr_{\mr \times \Gamma}: F \in L^2(\mr \times \p X )\},
\end{gather*}
 and $\mcs\restr_{\mr\times \Gamma}$ will denote the operator defined by
\begin{gather}
\begin{gathered}
\mcsg: L^2(\mr \times \Gamma) \longrightarrow L^2(\mr \times \Gamma) \\
F \longmapsto  (\mcs F)\restr_{\mr\times \Gamma}.
\end{gathered}\label{rest-mcs}
\end{gather}

Since we assume we know $\mca(\la)\restr_{\Gamma}$ for every $\la,$  then in view of \eqref{relation} we may assume we know $\mcsg.$ 
We shall prove
\begin{thm}\label{rangedet} Let $\Gamma\subset \p X$ be an open subset such that $\p X\setminus \Gamma$  does not have empty interior.  The space
\begin{gather*}
\mcm(\Gamma)^\pm=\{ \mcr_\pm(0,f)\restr_{\mr \times \Gamma}: f \in L^2_{ac}(X)\},
\end{gather*}
equipped with norm $\mcn$ defined by
\begin{gather}
\mcn\left(\mcr_\pm(0,f)\restr_{\mr \times \Gamma}\right)=||f||_{L^2(X)}. \label{defnorm}
\end{gather}
is a Hilbert space determined by $\mcsg.$
\end{thm}
\begin{proof}  We shall work with the forward radiation field.  The proof of the result for $\mcr_-$ is identical. Since $\mcr_+$ is linear, the triangle inequality for the $L^2(X)$-norm implies that $\mcn$ is a norm, and that
\begin{gather*}
\lan \mcr_+(0,f)\restr_{\mr\times \Gamma},\mcr_+(0,h)\restr_{\mr\times \Gamma} \ran_{\mcn}= \lan f,h\ran_{L^2(X)}
\end{gather*}
is an inner product.  Since $\mcr_+$ is continuous and $L_{ac}^2(X)$ is complete,  it follows that $(\mcm(\Gamma),\mcn)$ is a Hilbert space.  We need to show that it is determined by $\mcsg,$  and we begin by observing, as in \cite{sa}, that  the symmetry of the wave equation under time reversal gives that
\begin{gather}
\begin{gathered}
\mcr_+(f_1,f_2)(-s,y)=\mcr_+(f_1,0)(-s,y)+\mcr_+(0,f_2)(-s,y)= -\mcr_-(f_1,0)(s,y)+\mcr_-(0,f_2)(s,y).
\end{gathered}\label{symmetry1}
\end{gather}
In particular, if $F(s,y)=\mcr_+(f_1,f_2)(s,y),$  and if we denote $F^*(s,y)=F(-s,y),$
 then applying $\mcs$ to the second equality of \eqref{symmetry1}, we obtain
\begin{gather}
\mcs F^*(s,y)= -\mcr_+(f_1,0)+ \mcr_+(0,f_2)=\mcr_+(-f_1,f_2), \label{symmetry2}
\end{gather}
and hence we deduce that
\begin{gather}
\begin{gathered}
\text{ if } F(s,y)=\mcr_+(f_1,f_2), \text{ then } \ha(F+\mcs F^*)=\mcr_+(0,f_2), \text{ and } \ha(F-\mcs F^*)=\mcr_+(f_1,0),
\end{gathered}\label{symmetry3}
\end{gather}
and therefore,
\begin{gather}
\begin{gathered}
F(s,y)=\mcr_+(0,f_2), \text{ if and only }  F=\mcs F^*,\\
 F(s,y)=\mcr_+(f_1,0), \text{ if and only if }  F=-\mcs F^*.
\end{gathered}\label{symmetry3A}
\end{gather}

   Since $\mcr_+$ is an isometry, then for any $F \in L^2(\mr \times \Gamma)$ there exists $(f,h) \in E_{ac}(X)$ such that $\mcr(f,h)=F.$ But in view of  the first identity in \eqref{symmetry3} we have 
\begin{gather}
\ha(F+\mcs F^*)|_{\mr\times \Gamma}= \mcr_+(0,h)|_{\mr \times \Gamma}. \label{mcl1}
\end{gather}
We shall denote
\begin{gather}
\begin{gathered}
\mcl : L^2(\mr \times \Gamma) \longrightarrow L^2(\mr \times \Gamma) \\
F \longmapsto \ha(F+\mcs F^*)|_{\mr \times \Gamma}.
\end{gathered}\label{mapl}
\end{gather}
Since $\mcs$ is unitary, it follows that $||\mcl|| \leq 1.$  In fact, the map $\mcl$ gives more information:
\begin{lemma}\label{l2norm} If $F=\mcr_+(f,h)$ is supported in $\mr \times \overline{\Gamma},$ then $||h||_{L^2(X)}$ is determined by $\mcs\restr_{\mr \times \Gamma}.$
\end{lemma}
\begin{proof}  If $F(s,y)=\mcr_+(f,h)\in L^2(\mr \times \Gamma),$ since $\mcr_+$ is unitary, then according to \eqref{symmetry3}
\begin{gather*}
\lan F, (F + (\mcs F^*)\restr_{\mr \times \Gamma})\ran= 
\lan F, (F + \mcs F^*)\ran= \lan \mcr_+(f,h), \mcr_+(-f,h)\ran= \\ ||h||_{L^2(X)}^2-(||d_g f||_{L^2(X)}^2-\frac{n^2}{4}||f||_{L^2(X)}^2)
\end{gather*}
On the other hand, $||F||_{L^2(\mr\times \Gamma)}= ||h||_{L^2(X)}^2+(||d_g f||_{L^2(X)}^2-\frac{n^2}{4}||f||_{L^2(X)}^2),$ therefore
\begin{gather*}
\ha ||F||^2+\ha \lan F,(F + (\mcs F^*)\restr_{\mr \times \Gamma})\ran= ||h||^2.
\end{gather*}
\end{proof}

We know from \eqref{unitary} that given $F\in L^2(\mr\times \Gamma)$ there exists $(f,h)\in E_{ac}(X)$ such that $\mcr_+(f,h)=F.$  We can say the following about such initial data 

\begin{prop}\label{ah} Let $\Gamma \subset \p X$ be an open subset such that $\p X\setminus \Gamma$ contains an open set $\mco,$ and $h\in L_{ac}^2(X).$ Then there exists at most one $f$ such that
$(f,0) \in E_{ac}(X)$ and $\mcr_+(f,h)$ is supported in $\mr \times \Gamma.$  Moreover, the set
\begin{gather*}
\mcc(\Gamma)=\{ h \in L^2_{ac}(X): \text{ there exists }  (f,0) \in E_{ac}(X) \text{ such that }  \mcr_+(f,h)(s,y)=0, \; y\in \p X \setminus \Gamma\}
\end{gather*}
 is dense in $L^2_{ac}(X).$ 
\end{prop}
\begin{proof}  First, if $\mcr_+(f_1,h)$ and $\mcr_+(f_2,h)$ are supported in $\mr \times \overline{\Gamma},$ then $\mcr_+(f_1-f_2,0)$ is supported in 
$\mr \times \overline{\Gamma},$ but this implies that $\mcr_+(f_1-f_2,0)=0$ in $\mr \times \mco,$ and so Corollarry \ref{uniqueinf} implies that $f_1=f_2.$ 

Let $v\in L^2_{ac}(X)$ and assume that $\lan v,h\ran_{L^2(X)}=0$ for all $h \in \mcc(\Gamma).$ Since $\mcr_+$ is unitary, then for all
$(f,0)\in E_{ac}(X),$
\begin{gather*}
\lan v,h\ran_{L^2(X)}=\lan \mcr_+(0,v), \mcr_+(f,h)\ran_{L^2(\mr \times \p X)}
\end{gather*}
Since $h\in \mcc(\Gamma),$ is arbitrary, it follows that
\begin{gather*}
\lan \mcr_+(0,v),F\ran_{L^2(\mr \times \p X)}=0 \text{ for all } F \in L^2(\mr \times \Gamma).
\end{gather*}
Hence $\mcr_+(0,v)=0$ on $\mr \times \Gamma$ and by Corollary \ref{uniqueinf}, $v=0.$
\end{proof}

\begin{lemma}\label{injective} If $\Gamma \subset \p X$ is open and $\p X\setminus \Gamma$ contains an open subset, then the map $\mcl$ is injective and has dense range.
\end{lemma}
\begin{proof} If $F=\mcr_+(f,h),$ is supported in $\mr \times \ogamma,$ then $\mcl F=\mcr_+(0,h)|_{\mr \times \Gamma}.$   If $\mcl F=0$ then  
$\mcr_+(0,h)=0$ on $\mr \times \Gamma.$  It follows from Corollary \ref{uniqueinf}  that $h=0,$ and hence $F=\mcr(f,0).$  Since there exists an open subset $\mco \subset (\p X\setminus \Gamma),$ and $F$ is supported in $\mr \times \overline{\Gamma},$ it follows that $F=\mcr_+(f,0)=0$ in $\mr\times \mco,$ and again by Corollary \ref{uniqueinf}, $f =0$ and so $F=0.$
  
  Now we prove that its range is dense.  Let $H\in L^2(\mr \times \Gamma)$ be orthogonal to the range of $\mcl.$ In this case, since $H$ is supported in $\mr\times \overline{\Gamma},$ then  for every $F\in L^2(\mr \times \Gamma),$
\begin{gather*}
0=\lan H, (F+\mcs F^*)|_{\mr\times \Gamma}\ran_{L^2(\mr\times \Gamma)}=\lan H, F+\mcs F^*\ran_{L^2(\mr\times \Gamma)}.
\end{gather*}
 Since $\mcs$ is unitary its adjoint is  its inverse $\mcs^{-1}.$ Hence
\begin{gather*}
0=\lan H, F+\mcs F^*\ran_{L^2(\mr\times \Gamma)}=\lan H, F\ran_{L^2(\mr\times \Gamma)}+\lan \mcs^{-1} H, F^*\ran_{L^2(\mr\times \Gamma)}= \\
\lan H, F\ran+\lan (\mcs^{-1} H)^*, F\ran_{L^2(\mr\times \Gamma)}= \lan H+(\mcs^{-1} H)^*, F\ran_{L^2(\mr\times \Gamma)}
\end{gather*}
  Therefore  $H+(\mcs^{-1} H)^*|_{\mr\times \Gamma}=0.$   This implies that $H^*+(\mcs^{-1} H)|_{\mr\times \Gamma}=0.$   If
  $H=\mcr_+(f_1,f_2),$ then by arguing as above, we see that $H^*+\mcs^{-1} H=2\mcr_-(0,f_2).$  By Corollary \ref{uniqueinf}, $f_2=0.$  But since $H$ is supported in $\Gamma$ and $\p X \setminus \Gamma$ contains an open subset,  we deduce again from Corollary \ref{uniqueinf} that $f_1=0,$ and thus $H=0.$ 
\end{proof}

We shall denote
\begin{gather}
\mcf(\Gamma)=\mcl(\lcg)=\{ \mcr_+(0,f)\restr_{\mr \times \Gamma}: f \in \mcc(\Gamma)\},  \label{range}
\end{gather}
and equip $\mcf(\Gamma)$ with the norm given by Lemma \ref{l2norm}
\begin{gather*}
\mcn(\mcr_+(0,f))=||f||_{L^2(X)}.
\end{gather*}
$(\mcf(\Gamma),\mcn)$ is a normed vector space, and since $\mcc(\Gamma)$ is dense in $L^2(X),$ $\mcf(\Gamma)$ is dense in $(\mcm(\Gamma),\mcn).$ 
Hence $(\mcm(\Gamma), \mcn)$ is the completion of $(\mcf(\Gamma),\mcn)$ into a Hilbert space, and therefore it is determined by $\mcsg.$   Notice that the completion of $\mcf(\Gamma)$ with the $L^2(\mr\times \Gamma)$-norm is $L^2(\mr\times \Gamma).$ But
\begin{gather*}
||\mcr_+(0,h)\restr_{(\mr\times \Gamma)}||_{L^2(\mr \times \Gamma)} \leq ||h||_{L^2(X)},
\end{gather*}
and hence, $\mcn$ is a stronger norm and $(\mcm(\Gamma),\mcn)$ is a smaller space.  This ends the proof of Theorem \ref{rangedet}.
\end{proof}

\section{Proof of Theorem \ref{INV}}

 As in \cite{sa} and \cite{colin3}, first we construct an isometry between neighborhoods of $\Gamma$ which realizes \eqref{diffeo}.  Then we apply the result of Kurylev and Lassas \cite{KuLa} on the reconstruction of a Riemannian manifold from the Dirichlet-to-Neumann  on part of the boundary to show that the diffeomorphism between neighborhoods of $\Gamma$  can be extended to an isometry between the two manifolds.   

First recall that the maps given by \eqref{metric}
\begin{gather*}
\Psi_{j,\eps}:  [0,\eps) \times \p X_j   \longrightarrow  U_{\eps,j} \subset X_j \\
(x_j,y) \mapsto z 
\end{gather*}
put the metrics $g_j$ in the form
\begin{gather*}
\Psi_j^* g_j= \frac{dx_j^2}{x_j^2}+ \frac{h_j(x_j,y,dy)}{x_j^2}.
\end{gather*}
As shown in \cite{Gra}, and discussed in section \ref{intro},  $\Psi_{j,\eps}(x_j,y)$ is the point obtained by flowing the integral curve of $\nabla_{x_j^2 g_j} x_j$ emanating from the point $y\in \p X_j$ by $x_j$ units of time.  So $x_j$ is the arc-length along the geodesics of $x_j^2 g_j$ normal to $\p X_j.$   We can pick $\eps$ small such that both maps are diffeomorphisms. For $z=(x,y) \in \Gamma\times [0,\eps),$ $x(z)$ is the distance from $z$  to $\Gamma$ with respect to the metric $x_j^2 g_j.$  In particular, when $y\in \Gamma,$
\begin{gather}
\Psi_j^*g_j= \frac{dx^2}{x^2}+ \frac{h_j(x,y,dy)}{x^2}, \;\ y\in \Gamma \label{fomet}
\end{gather}
Our first result will be
\begin{prop}\label{propo}
Let $(X_1,g_1),$  $(X_2,g_2)$ and $\Gamma$ satisfy the hypotheses of Theorem \eqref{INV}.  Let $\mcr_{j,\pm}(s,y,x',y')$ denote the kernels of the radiation fields acting on pairs of the form $(0,f),$ $f\in L_{ac}^2(X).$ Then there exists $\epsilon>0$ such that in the product decomposition $[0,\epsilon)\times \Gamma$ where \eqref{fomet} holds, $h_1(x,y,dy)=h_2(x,y,dy)$ and
\begin{gather}
\mcr_{1,\pm}(s,y,x',y')= \mcr_{2,\pm}(s,y,x',y'), \text{ if } y,y'\in \Gamma, \;\ x'<\eps.\label{equalradf}
\end{gather}
\end{prop}
\begin{proof}
The proof of Proposition \ref{propo} is an adaptation of the Boundary Control Method to this setting. As in \cite{sa}, pick $x_1<\eps,$ and consider the spaces
\begin{gather*}
\mcm_{x_1}^+(\Gamma)= \{ F \in \mcm^+(\Gamma): F(s,y)=0, s \leq \log x_1 \}, \\
\mcm_{x_1}^-(\Gamma)= \{ F \in \mcm^-(\Gamma): F(s,y)=0, s \geq - \log x_1 \},
\end{gather*}
and let
\begin{gather}
\begin{gathered}
\mcp_{x_1}^+: \mcm^+(\Gamma) \longrightarrow \mcm_{x_1}^+(\Gamma), \text{ and } \\
\mcp_{x_1}^-: \mcm^-(\Gamma) \longrightarrow \mcm_{x_1}^-(\Gamma)
\end{gathered}\label{projectors}
\end{gather}
denote the orthogonal projections with respect to the norm $\mcn$ defined in \eqref{defnorm}. Since $\mcm^{\pm}(\Gamma)$ and $\mcm_{x_1}^{\pm}(\Gamma)$ are determined  by $\mcsg,$ the projections $\mcp_{x_1}^{\pm}$ are also determined by $\mcsg.$

In view of finite speed of propagation and Theorem \ref{support-local}
\begin{gather*}
\mcm_{x_1}^+(\Gamma)= \{ \mcr_+(0,h)\restr_{\mr\times \Gamma}: h \in L_{ac}^2(X), \; h(z)=0, \; z \in \mcd_{\log x_1}(\Gamma) \}, \\
\mcm_{x_1}^-(\Gamma)= \{ \mcr_-(0,h)\restr_{\mr\times \Gamma}: h \in L_{ac}^2(X), \; h(z)=0, \; z \in \mcd_{\log x_1}(\Gamma) \}.
\end{gather*}

As in \cite{sa}, the key to proving Proposition \ref{propo} is to understand the effect of the projectors $\mcp_{x_1}^\pm$ on the initial data.  First we deal with the case of no eigenvalues.  In this case, $L^2(X)=L^2_{ac}(X).$

\begin{lemma}\label{87}
Let $(X,g)$ be an asymptotic hyperbolic manifold such that $\Delta_{g}$ has no eigenvalues. Let $x$ be such that \eqref{metric} holds in $(0,\epsilon)\times\partial X$. For $x_1\in (0,\epsilon)$, let $\mathcal{P}_{x_1}^{-}$ denote the orthogonal projector defined in \eqref{projectors}.
Let $\chi_{x_1}$ be the characteristic function of the set $X_{x_1}= X\setminus \mcd_{\log x_1}(\Gamma) $. Then for every $f\in L^2_{\text{ac}}(X)=L^2(X),$
\begin{gather*}
\mathcal{P}_{ x_1}^-(\mathcal{R}_{-}(0,f)\restr_{\mr \times \Gamma}) =\mathcal{R}_{-}(0,\chi_{x_1}f)\restr_{\mr \times \Gamma}.
\end{gather*}
\end{lemma}
\begin{proof}   Then since $\mcp_{x_1}^-$ is a projector, there exists
$f_{x_1}\in L^2(X)$ such that $\mathcal{P}_{ x_1}^-(\mathcal{R}_{-}(0,f)\restr_{\mr \times \Gamma}) =\mathcal{R}_{-}(0,f_{x_1})\restr_{\mr \times \Gamma},$ and
for every $h\in L^2(X)$ supported in $X_{x_1},$
\begin{gather*}
\lan \mcr_-(0, f_{x_1})\restr_{\mr \times \Gamma},  \mcr_-(0, h)\restr_{\mr \times \Gamma}\ran_{\mcn}=
\lan f_{x_1}, h\ran_{L^2(X)}=\lan f,h\ran_{L^2(X)}.
\end{gather*}
Hence $f_{x_1}=\chi_{x_1} f.$  
\end{proof}
Next we analyze the singularities of the wave which produces $\mcr_{\pm}(0,\chi_{x_1} f).$   Notice that if $f=0$ in $\mcd_{\log x_1}(\Gamma),$ and if $x\leq x_1$ and $y\in \Gamma,$  then $d_g( (x,y), \Gamma_{x_1})= \log (\frac{x_1}{x}),$ and hence, $f(x,y)=0.$

If $F=\mcr_-(0,f)\restr_{\mr \times \Gamma},$ $f\in L^2(X),$
one can take $F$ and $f$ smooth by taking convolution with $\phi\in C_{0}^{\infty}$  even.  Indeed, if $\psi_1$ such that $\displaystyle\hat{\phi}(\lambda)=\psi_{1}(\lambda^2),$ and $G=\phi\ast F=\mcr_-(0,h)\restr_{\mr \times \Gamma}$ with  $h=\psi_{1}\left(\Delta_{g}-\frac{n^2}{4}\right)\mathcal{R}_{-}^{-1}F.$ The point is that the singularities of $\chi_{x_1} f$ at  $\Gamma_{x_1}$ produce the singularities of $\mcr_+(0,\chi_{x_1} f)$ at $(\log x_1,y),$ $y \in \Gamma.$
By expanding the solution to \eqref{wave} with initial data, $(0,\chi_{x_1}f)$ we obtain, 
\begin{lemma}\label{lema}
 Let $x$ be a defining function of $\partial X$ such that \eqref{metric} holds. Let 
$F\in \mcm^-(\Gamma)=\mcr_-(0,f)\restr_{\mr \times \Gamma}$ with $F$ and $f$ smooth.  Then there exists $\epsilon>0$ such that for any $x_1\in (0,\epsilon)$, any $F\in\mathcal{M}^b,$  and for  $|s-\log x_1|$ small enough,
\begin{gather}
\mathcal{R}_{+}\mathcal{R}_{-}^{-1}(\mathcal{P}_{x_1}^b F)(s,y)=\frac{1}{2}x_{1}^{-n/2}f(x_1,y)\frac{|h|^{1/4}(x_1,y)}{|h|^{1/4}(0,y)}(s-\log x_1)_{+}^{0}+\text{smoother terms}, \label{singexp}
\end{gather}
\end{lemma}
We refer the reader to the proof of Lemma 8.9 of \cite{sa}.

Notice that the left hand side of \eqref{singexp} is determined by $\mcsg,$ so the right hand side is also determined by $\mcsg.$  By assumption--since the scattering matrices are equal--  $h_{0,1}=h_{0,2}$ on $\Gamma.$ Therefore $|h_{1}|(0,y)=|h_2|(0,y),$ $y\in \Gamma$ and we obtain the following result:
\begin{cor}\label{82}
Let $(X_{1},g_1)$ and $(X_2,g_2)$ be asymptotically hyperbolic manifolds satisfying the hypothesis of Theorem \ref{INV}. Moreover, assume that $\Delta_{g_j}$, $j=1,2$, have no eigenvalues. Let $\mathcal{R}_{j,\pm},$ $j=1,2,$ denote the corresponding forward or backward radiation fields defined in coordinates in which \eqref{metric} holds. Then there exists an $\epsilon>0$ such that, for $(x,y)\in [0,\epsilon)\times \Gamma$,
\begin{equation}\label{812}
\begin{gathered}
|h_{1}|^{1/4}(x,y)\mathcal{R}_{1,-}^{-1}F(x,y)=|h_{2}|^{1/4}(x,y)\mathcal{R}_{2,-}^{-1}F(x,y),\quad \forall F\in\mathcal{M}^-(\Gamma),\\
|h_{1}|^{1/4}(x,y)\mathcal{R}_{1,+}^{-1}F(x,y)=|h_{2}|^{1/4}(x,y)\mathcal{R}_{2,+}^{-1}F(x,y),\quad \forall F\in\mathcal{M}^+(\Gamma).
\end{gathered}
\end{equation}
\end{cor}

Proposition \ref{propo} easily follows from this result. Indeed,  since 
\begin{equation}\label{813}
\mathcal{R}_{j,-}^{-1}\left(\frac{\partial^2}{\partial s^2}F\right)=\left(\Delta_{g_j}-\frac{n^2}{4}\right)\mathcal{R}_{j,-}^{-1}F.
\end{equation}
if we apply Corollary \ref{82} to $\p_s^2 F,$ we obtain
\begin{equation}\label{814}
|h_1|^{1/4}(x,y)\left(\Delta_{g_1}-\frac{n^2}{4}\right)\mathcal{R}_{1,-}^{-1}F(x,y)=|h_2|^{1/4}(x,y)\left(\Delta_{g_2}-\frac{n^2}{4}\right)\mathcal{R}_{2,-}^{-1}F(x,y).
\end{equation}

Set $\mathcal{R}_{1,-}^{-1}F=f.$ Since $F$ is arbitrary and the metrics have no eigenvalues, the equations \eqref{812} and \eqref{814} give
\begin{equation}\label{815}
|h_1|^{1/4}(x,y)\left(\Delta_{g_1}-\frac{n^2}{4}\right)f(x,y)=|h_2|^{1/4}(x,y)\left(\Delta_{g_2}-\frac{n^2}{4}\right)\frac{|h_1|^{1/4}(x,y)}{|h_2|^{1/4}(x,y)}f(x,y),
\end{equation}
for all $\displaystyle f\in C^{\infty}\left( (0,\epsilon)\times \Gamma\right).$ Therefore the operators on both sides of \eqref{815} are equal. In particular, the coefficients of the principal parts of $\Delta_{g_1}$ are equal to those of $\Delta_{g_2}$, and hence the tensors $h_1$ and $h_2$ from \eqref{metric} are equal. This  proves that
\begin{gather*}
\mcr_{1,-}^{-1}(s,y,x',y')=\mcr_{2,-}^{-1}(s,y,x',y'), \;\ y,y'\in \Gamma, \;\ x'\in [0,\eps),
\end{gather*}
and of course the same holds for the forward radiation field. Since $\mcr_\pm$ are unitary, $\mcr_{\pm}^{-1}=\mcr_{\pm}^*,$ and hence this determines the kernel of $\mcr_\pm.$
This proves  Proposition \eqref{propo} in the case of no eigenvalues.

Now we remove the assumption that there are no eigenvalues.
The only poles of the resolvent 
\begin{gather*}
\displaystyle R(\novt+i\lambda)=(\Delta_{g}-\nsq-\lambda^2)^{-1} \text{ in } \{\Im \lambda<0\}
\end{gather*}
 correspond to the finitely many eigenvalues of $\Delta_{g}$. Proposition 3.6 of \cite{GraZwo} states that if $\lambda_{0}\in i\mathbb{R}_{-}$ is such that $n^2/4+\lambda_{0}^2$ is an eigenvalue of $\Delta_{g}$, then the scattering matrix has a pole at $\lambda_{0}$ and its residue is given by
\begin{equation}\label{832}
\text{Res}_{\lambda_{0}}A(\lambda)=\left\{\begin{array}{l}
\Pi_{\lambda_0},\quad \text{if }-i\lambda_0\not\in \frac{\mathbb{N}}{2},\\
\Pi_{\lambda_0}-P_{l},\quad \text{if }-i\lambda_{0}=\frac{l}{2},\quad l\in\mathbb{N},
\end{array}\right.
\end{equation}
where $P_{l}$ is a differential operator whose coefficients depend on derivatives of the tensor $h$ at $\p X,$  and the Schwartz kernel of $\Pi_{\lambda_0}$ is 
\begin{equation}\label{833}
K(\Pi_{\lambda_0})(y,y')=-2i\lambda_{0}\sum_{j=1}^{N_{0}}\phi_{j}^{0}\otimes\phi_{j}^{0}(y,y'),
\end{equation}
where $\phi_{j}^{0}(y)$ is defined by
$$\phi_{j}^{0}(y)=x^{-n/2-\lambda_0}\phi_{j}(x,y)|_{x=0}.$$
Here $N_{0}$ is the multiplicity of the eigenvalue $n^2/4+\lambda_{0}^{2}$, and $\phi_j$, $1\leq j\leq N_{0},$ are the corresponding orthonormalized eigenfunctions.

Since $A_{1,\Gamma}=A_{2,\Gamma},$ $\la \in \mr\setminus 0,$  it follows from Theorem \ref{JoSa2} that  in coordinates where \eqref{fomet} is satisfied, all derivatives of $h_1$ and $h_2$ agree at $x=0$ on $\Gamma$. Therefore the operators $P_{l}$ in \eqref{832} are the same in $\Gamma$. Then \eqref{832}, \eqref{833}, and the meromorphic continuation of the scattering matrix show that $\Delta_{g_1}$ and $\Delta_{g_2}$ have the same eigenvalues with the same multiplicity. Moreover, \eqref{833} implies that if $\phi_j$ and $\psi_j$, $1\leq j\leq N_{0}$, are orthonormal sets of eigenfunctions of $\Delta_{g_1}$ and $\Delta_{g_2}$, respectively, corresponding to the eigenvalue $n^2/4+\lambda_{0}^2$, then there exists a constant orthogonal $(N_{0}\times N_{0})-$matrix $A$ such that $\displaystyle \Phi^{0}|_{\Gamma}=A\Psi^{0}|_{\Gamma},$ where $\displaystyle (\Phi^{0})^{T}=(\phi_{1}^{0},\phi_{2}^{0},...,\phi_{N_{0}}^{0})$ and $\displaystyle (\Psi^{0})^{T}=(\psi_{1}^{0},\psi_{2}^{0},...,\psi_{N_0}^{0}).$ So by redefining one set of eigenfunctions from, let us say, $\Psi$ to $A\Psi$, where $\Psi^{T}=(\psi_{1},\psi_{2},...,\psi_{N_0}),$ we may assume that
\begin{equation}\label{834}
\phi_{j}^{0}(y)=\psi_{j}^{0}(y),\quad y\in\Gamma\quad j=1,2,..,N_{0}.
\end{equation}

Note that this does not change the orthonormality of the eigenfunctions in $X_2$ because $A$ is orthogonal. Denote the eigenvalues of $\Delta_{g_1}$ and $\Delta_{g_2}$, which we know are equal, by
\begin{equation}\label{835}
\mu_{j}=\frac{n^2}{4}+\lambda_{j}^2,\quad \lambda_{j}\in i\mathbb{R}_{-},\quad 1\leq j\leq N.
\end{equation}
They are also ordered so that $\mu_{1}\leq \mu_{2}\leq \cdot\cdot\cdot\leq \mu_{N}.$

 Again, we use that the singularities of $\chi_{x_1} f$ at  $\Gamma_{x_1}$ produce the singularities of $\mcr_+(0,\chi_{x_1} f)$ at $(\log x_1,y),$ $y \in \Gamma$ and
expand the solution to \eqref{wave} with initial data, $(0,\chi_{x_1}f).$  However, in this case $L^2(X)\not= L_{ac}^2(X)$ and hence Lemma \ref{87} is not valid, and we have to replace it by the following 
\begin{lemma}\label{87E} Let $(X,g)$ be an asymptotic hyperbolic manifold and let $\phi_j,$ $1\leq j \leq N,$ denote the orthonormal set of eigenfunctions of $\Delta_g.$  Let $x$ be such that \eqref{metric} holds in $(0,\epsilon)\times\partial X$. For $x_1\in (0,\epsilon)$, let $\mathcal{P}_{x_1}^{-}$ denote the orthogonal projector defined in \eqref{projectors}.
Let $\chi_{x_1}$ be the characteristic function of the set $X_{x_1}= X\setminus \mcd_{\log x_1}(\Gamma)$. There exists $\eps_0$ such that if
$\eps<\eps_0,$ then for every $f\in L^2_{\text{ac}}(X)$ there exist $\alpha(x_1,f),$ which is a linear function of $f,$ such that
\begin{gather*}
\mathcal{P}_{ x_1}^-(\mathcal{R}_{-}(0,f)\restr_{\mr \times \Gamma}) =\left.\mathcal{R}_{-}\left(0,\chi_{x_1}(f-\sum_{j=1}^N \alpha_j(x_1,f) \phi_j )\right)\right|_{\mr \times \Gamma}.
\end{gather*}
\end{lemma}
\begin{proof}  Let $h \in L_{ac}^2(X)$ be supported in $X_{x_1}.$ Then since $\mcp_{x_1}^-$ is a projector, there exists
$f_{x_1}\in L_{ac}^2(X),$ supported in $X_{x_1},$  such that $\mathcal{P}_{ x_1}^-(\mathcal{R}_{-}(0,f)\restr_{\mr \times \Gamma}) =\mathcal{R}_{-}(0,f_{x_1})\restr_{\mr \times \Gamma},$ and
for every $h\in L_{ac}^2(X)$ supported in $X_{x_1},$
\begin{gather*}
\lan \mcr_-(0, f_{x_1})\restr_{\mr \times \Gamma},  \mcr_-(0, h)\restr_{\mr \times \Gamma}\ran_{\mcn}=
\lan f_{x_1}, h\ran_{L^2(X)}=\lan f,h\ran_{L^2(X)}.
\end{gather*}
Hence $\lan (f_{x_1}-f),h\ran=0$ for all $h\in C_0^\infty(X)\cap L^2_{ac}(X)$ supported in $X_1.$   Therefore,
$f_{x_1}-f= \chi_{x_1} \sum_{j=1}^N \alpha_j \phi_j.$ Since $f_{x_1}\in L^2_{ac}(X),$ $\lan f_{x_1},\chi_{x_1}\phi_j\ran=0,$ and hence
\begin{gather*}
\lan f, \chi_{x_1} \phi_k\ran = \sum_{j=1}^N \alpha_j \lan \chi_{x_1} \phi_j, \chi_{x_1} \phi_k\ran.
\end{gather*}
This gives a linear system of equations 
\begin{gather*}
M \alpha= F, \;\ \alpha^T=(\alpha_1,...,\alpha_N), \;\ F^T=(F_1(x_1), ..., F_N(x_1)\}, \\ M_{jk}(x_1)=\lan \chi_{x_1} \phi_j, \chi_{x_1} \phi_k\ran, \;\ F_k(x_1)=\lan f, \chi_{x_1} \phi_k\ran.
\end{gather*}
Since eigenfunctions are orthonormal,  then for $x_1=0,$ $M_{jk}(0)=\del_{jk}.$  Therefore the system has a solution if $x_1<\eps_0,$ for certain $\eps_0,$ which does not depend of $f.$   Notice that, since $f\in L^2_{ac}(X),$ then for $x_1=0,$ $F_k(0)=0,$ and hence, $\alpha(0,f)=0.$
\end{proof}

As in \cite{sa}, we shall denote 
\begin{gather*}
T(x_1) f= \sum_j \alpha_j(x_1,f) \phi_j.
\end{gather*}
Since $\alpha(0,f)=0,$ $T(0)=0.$ Therefore one can pick $\eps$ small so that 
\begin{gather}
||T(x_1)||<\ha \text{ for } x_1<\eps. \label{smallT}
\end{gather}
In the case, Lemma \ref{lema} and Corollary \ref{82} have to be substituted by
\begin{lemma}\label{lemaE} Let $(X,g)$ be an asymptotically hyperbolic manifold, let $\phi_j,$ $1\leq j \leq N,$ denote the eigenfunctions of $\Delta_g.$ 
 Let $x$ be a defining function of $\partial X$ such that \eqref{metric} holds,  let $T(x_1)$ be defined as above and let
$F\in \mcm^-(\Gamma),$ $F=\mcr_-(0,f)\restr_{\mr \times \Gamma}$ with $F$ and $f$ smooth.  Then there exists $\epsilon>0$ such that for any $x_1\in (0,\epsilon)$,  and for  $|s-\log x_1|$ small enough,
\begin{gather}
\mathcal{R}_{+}\mathcal{R}_{-}^{-1}(\mathcal{P}_{x_1}^b F)(s,y)=\frac{1}{2}x_{1}^{-n/2}[(\id-T(x_1))f](x_1,y)\frac{|h|^{1/4}(x_1,y)}{|h|^{1/4}(0,y)}(s-\log x_1)_{+}^{0}+\text{smoother terms}, \label{singexpE}
\end{gather}
\end{lemma}

\begin{cor}\label{82E}
Let $(X_{1},g_1)$ and $(X_2,g_2)$ be asymptotically hyperbolic manifolds satisfying the hypothesis of Theorem \ref{INV}.  Let $\mathcal{R}_{j,\pm},$ $j=1,2,$ denote the corresponding forward or backward radiation fields defined in coordinates in which \eqref{metric} holds. Then there exists an $\epsilon>0$ such that, for $(x,y)\in [0,\epsilon)\times \Gamma$,
\begin{equation}\label{812E}
\begin{gathered}
|h_{1}|^{1/4}(x,y)(\id-T_1(x))\mathcal{R}_{1,-}^{-1}F(x,y)=|h_{2}|^{1/4}(x,y)\mathcal{R}_{2,-}^{-1}(\id-T_2(x))F(x,y),\quad \forall F\in\mathcal{M}^-(\Gamma),\\
|h_{1}|^{1/4}(x,y)(\id-T_1(x))\mathcal{R}_{1,+}^{-1}F(x,y)=|h_{2}|^{1/4}(x,y)(\id-T_2(x))\mathcal{R}_{2,+}^{-1}F(x,y),\quad \forall F\in\mathcal{M}^+(\Gamma).
\end{gathered}
\end{equation}
\end{cor}
Now, as in the case of no eigenvalues, we use \eqref{812E} to show that we can apply the same identity to $(\Delta_{g_j}-\nsq)f_j,$  instead of $f_j$ and so we have
\begin{gather}
|h_1|^\oq (\id-T_1(x))(\Delta_{g_1}-\nsq) f_1(x,y) =|h_2|^\oq(\id-T_2(x)) |h_2|^\oq (\Delta_{g_2}-\nsq) f_2(x,y).\label{817E}
\end{gather}

If we denote $\mcr_{j,-}^{-1}F(x,y)= f_j(x,y),$  and pick $\eps$ small so that \eqref{smallT} holds, we deduce from \eqref{812E} that if $(x,y) \in [0,\eps)\times \Gamma,$
\begin{gather*}
f_2(x,y) =(\id-T_2(x))^{-1}(\id-T_1(x)) \frac{|h_1|^\oq}{|h_2|^\oq} f_1(x,y)= \frac{|h_1|^\oq}{|h_2|^\oq} f_1(x,y)+ K(x) f_1(x,y),\\
(\Delta_{g_2}-\nsq) f_2(x,y)= (\id-T_2(x))^{-1}(\id-T_1(x)) \frac{|h_1|^\oq}{|h_2|^\oq}(\Delta_{g_1}-\nsq) f_1(x,y) = \\
 \frac{|h_1|^\oq}{|h_2|^\oq}(\Delta_{g_1}-\nsq) f_1(x,y) + \widetilde{K} f_1(x).
\end{gather*}
where $K$ and $\widetilde{K}$ are compact operators.    If one substitutes the first equation into the second, one obtains
\begin{gather*}
 \frac{|h_1|^\oq}{|h_2|^\oq}(\Delta_{g_1}-\nsq) f_1(x,y)- (\Delta_{g_2}-\nsq) \frac{|h_1|^\oq}{|h_2|^\oq} f_1(x,y) = \mathcal{K} f_1,
 \end{gather*}
 where $\mathcal{K}$ is a compact operator.  Since the operator on the left hand side is a differential operator, and the operator on the right hand side is compact, they both must be equal to zero. As above, we conclude that in coordinates $(x,y),$ the coefficients of the operators $\Delta_{g_1}$ are equal to those of $\Delta_{g_2}.$ Hence we must have $h_1(x,y,dy)=h_2(x,y,dy).$

 We still have to show \eqref{equalradf} holds in the case where eigenvalues exist.  Let $F \in \mcm^+(\Gamma),$ and let $f_j =\mcr_{j,+}^- F.$ Let $v_j$ satisfy \eqref{wave} with initial data $(0,f_j).$  Since $\Delta_{g_1}=\Delta_{g_2}$ in $(0,\eps)\times \Gamma,$ we have $P (v_1-v_2)=0,$ where $P$ is defined in \eqref{operatorP}.  Then by the proof of Theorem \ref{support-local}, we have $f_1=f_2$ in $(0,\eps)\times \Gamma.$  Since $F$ is arbitrary, \eqref{equalradf} follows.
 This ends the proof of Proposition \ref{propo}. 
 \end{proof}

  Finally, we will prove Theorem \ref{INV}.   The fact that the metrics are equal in these coordinates implies that there exist $V_{j,\eps}\subset X_j,$ $j=1,2,$ neighborhoods of $\Gamma$ and $C^\infty$ differomomrphisms
\begin{gather*}
\Psi_{j,\eps}: \Gamma \times [0,\eps) \longrightarrow V_{j,\eps}\subset X_j
\end{gather*} 
such that
\begin{equation}\label{86}
\Psi^{*}_{1,\epsilon}(g_{1}|_{V_{1,\epsilon}})=\Psi^{*}_{2,\epsilon}(g_{2}|_{V_{2,\epsilon}}).
\end{equation}
Therefore, the map
\begin{gather}
\Psi_{\eps}: V_{1,\eps} \longrightarrow V_{2,\eps}, \text{ defined by } \Psi_\eps=\Psi_{1,\eps} \circ \Psi_{2,\eps}^{-1},\label{defnpsie}
\end{gather}
satisfies

\begin{equation}\label{E87}
\Psi_\eps^*(g_2|_{U_{2,\epsilon}})=g_{1}|_{U_{1,\epsilon}}, \text{ and }  \Psi_\eps= \id \text{ on } \Gamma.
\end{equation}
We will show that  $\Psi_\eps$ can be extended to the whole manifolds, and to do that we will show that the fact that scattering matrices are equal at $\Gamma$ implies that there exist compact manifolds with boundary $X_{j,\eps}\subset X_j$ such that $\Gamma_\eps \subset X_{1,\eps}\cap X_{2,\eps}$ as manifolds and that the eigenvalues of $\Delta_{g_j}$ with Dirichlet data on $X_{j,\eps}$ are equal and the traces of the normal derivatives of the Dirichlet eigenfunctions of $\Delta_{g_j},$ $j=1,2$ coincide at $\Gamma.$ We shall then use  the following result due to Kurylev and Lassas \cite{KaKuLa} and \cite{KuLa}.
\begin{prop}\label{pkl}
Assume that $(Z_1,g_1)$ and $(Z_2,g_2)$ are smooth Riemannian manifolds with boundary $\partial Z_1$ and $\partial Z_2$, respectively. Assume that there are open sets $\Gamma_1\subset\partial Z_1$ and $\Gamma_2 \subset\partial Z_2$, such that the boundary spectral data on $\Gamma_2$ and $\Gamma_2$ coincide. Namely,  if $\alpha_{1,j},$ and $\vphi_{1,j}$ are the eigenvalues and the normalized Dirichlet eigenfunctions of $\Delta_{g_1}$ and 
 and $\alpha_{2,j},$ $\vphi_{2,j}$  are  the eigenvalues and the normalized Dirichlet eigenfunctions of $\Delta_{g_2},$
\begin{gather*}
\Gamma_1=\Gamma_2 \text{ as manifolds, and}\quad \alpha_{1,j}=\alpha_{2,j},\quad \partial_{\nu}\vphi_{1,j}|_{\Gamma_1}=\partial_{\nu}\tilde{\vphi}_{2,j}|_{\Gamma_2},\quad
j=1,2,...
\end{gather*} 
 Then there exists a diffeomorphism $\Psi: Z_1 \longrightarrow Z_2$ such that $\Psi^* \tilde{g}=g$ and $\Psi=\id$ on $\Gamma.$ 
\end{prop}

Note that $X_{j,\epsilon}=X_j\setminus \{x_j<\epsilon\}$ are smooth compact manifolds with boundary and that there are open sets $\Gamma_{1,\epsilon}\subset\partial X_{1,\epsilon}$, $\Gamma_{2,\epsilon}\subset\partial X_{2,\epsilon}$ which can be identified by the diffeomorphisms $\Psi_{j,\eps}$  with $\Gamma\times\{\epsilon\}=\Gamma_{\epsilon}.$ 
We think of $\Gamma_{\epsilon}$ as an open subset of the boundary smooth compact manifolds $X_{j,\eps}$. We have shown that  the fact that the two metrics have the same scattering matrix restricted to $\Gamma,$ implies that $\Delta_{g_1}=\Delta_{g_2}$ in $[0,\eps) \times \Gamma.$ 
The next step is to show that $(X_{j,\eps}, g_j)$, $j=1,2,$ have the same spectral data in $\Gamma_{\epsilon}.$ Then Proposition \eqref{pkl} shows that  for $\eps>0$ fixed, there exists a  $C^\infty$ diffeomorphism
$\wtpsi_\eps,$   such that
\begin{gather}
\begin{gathered}
\wtpsi_{\eps}: \overline{X}_{1,\eps} \longrightarrow \overline{X}_{2,\eps}, \;\ \wtpsi_{\eps}\restr_{\Gamma_\eps}=\id, \\
\Psi_\eps^* (g_2\restr_{\overline{X}_{2,\eps}})= g_1\restr_{\overline{X}_{1,\eps}}.
\end{gathered}\label{familymaps}
\end{gather}
Now one has to extend this diffeomorphism up to $X_1.$ Given a point $p \in  X_1\setminus X_{1,\eps}$ take the integral curve of $\nabla_{x_1^2 g_1} x_1$ starting at $p$ and flow along this curve until it hits $p_\eps\in \p X_{1,\eps}$ at time $s(p).$ In other words, $p_{\eps}= \exp( s(p) \nabla_{x_1^2 g_1} x_1)(p).$  Let $ q_\eps=\Psi_\eps(p_{\eps})\in \p X_{2,\eps}.$ In view of \eqref{familymaps},  and since $g_j= \frac{dx_j^2}{x_j^2} + \frac{h_j(x,y,dy)}{x_j^2},$ $j=1,2,$ near $\p X_{j,\eps},$  it follows that
$d\Psi_{\eps}( \nabla_{x_1^2 g_1} x_1)\restr_{\p X_{2,\eps}}=\nabla_{x_2^2g_2} x_2\restr_{\p X_{2,\eps}}.$ Now follow the integral curve of $\nabla_{x_2^2 g_2} x_2$   starting at $q_\eps,$ and flow $s(p)$ units along this curve away from $X_{2,\eps}.$  In other words, in $X_j \setminus X_{j,\eps}$ the map $\widetilde{\Psi}_\eps$ is given by
\begin{gather*}
\exp( -s(p) \nabla_{x_2^2 g_2} x_2) \circ \Psi_\eps \circ \exp( s(p) \nabla_{x_1^2 g_1} x_1)= \widetilde{\Psi}_\eps.
\end{gather*}
 This gives a map,
\begin{gather}
\wtpsi: X_1 \longrightarrow X_2,  \;\ \wtpsi\restr_{\Gamma}=\id \text{ and that }
{\wtpsi}^* g_2= g_1.
\end{gather}

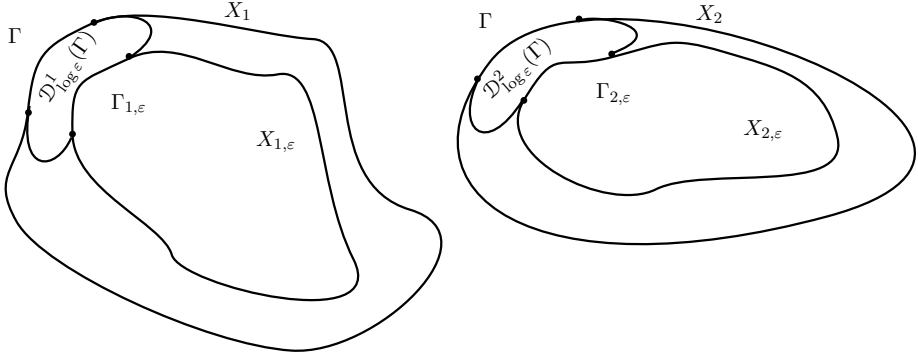
\begin{figure}[htb!]
% Generated with LaTeXDraw 2.0.8
% Thu Jun 06 10:43:04 EDT 2013
% \usepackage[usenames,dvipsnames]{pstricks}
% \usepackage{epsfig}
% \usepackage{pst-grad} % For gradients
% \usepackage{pst-plot} % For axes
\scalebox{.75} % Change this value to rescale the drawing.
{
\begin{pspicture}(0,-3.8009243)(20.42,3.8009243)
\usefont{T1}{ptm}{m}{n}
\rput(4.75,2.3440757){$X_1$}
\usefont{T1}{ptm}{m}{n}
\rput(5.43,0.024075773){$X_{1,\eps}$}
\psdots[dotsize=0.12](1.04,0.5390758)
\psdots[dotsize=0.12](2.2,2.1390758)
\psdots[dotsize=0.12](1.82,0.15907577)
\psdots[dotsize=0.12](2.82,1.5390757)
\usefont{T1}{ptm}{m}{n}
\rput(0.8,1.9040757){$\Gamma$}
\psbezier[linewidth=0.04](1.06,0.6590758)(0.84,-0.5809242)(1.54,-0.5057727)(1.8,0.09703432)
\psbezier[linewidth=0.04](2.2,2.1222267)(2.8435,2.4390757)(3.76,2.1257086)(2.7998102,1.4941869)
\psbezier[linewidth=0.04](9.22,1.4390758)(7.42,-0.90277267)(10.0,-2.7009242)(15.2,-1.3009242)(20.4,0.09907577)(11.02,3.7809243)(9.22,1.4390758)
\psdots[dotsize=0.12](9.0,1.1390758)
\psdots[dotsize=0.12](10.8,2.1990757)
\usefont{T1}{ptm}{m}{n}
\rput(9.14,2.1640759){$\Gamma$}
\psdots[dotsize=0.12](9.82,0.75907576)
\psdots[dotsize=0.12](11.38,1.5790758)
\psbezier[linewidth=0.04](9.02,1.1390758)(8.58,0.25907576)(9.14,-0.28092423)(9.8,0.75907576)
\psbezier[linewidth=0.04](10.82,2.1790757)(11.2024975,2.2998993)(12.52,1.9990758)(11.347266,1.5790758)
\usefont{T1}{ptm}{m}{n}
\rput(13.11,2.2840757){$X_2$}
\usefont{T1}{ptm}{m}{n}
\rput(14.07,0.22407578){$X_{2,\eps}$}
\usefont{T1}{ptm}{m}{n}
\rput{55.042995}(1.845145,-0.8830941){\rput(1.77,1.3440758){$\mcd_{\log \eps}^1(\Gamma)$}}
\usefont{T1}{ptm}{m}{n}
\rput(2.82,0.6840758){$\Gamma_{1,\eps}$}
\usefont{T1}{ptm}{m}{n}
\rput(11.42,0.8640758){$\Gamma_{2,\eps}$}
\usefont{T1}{ptm}{m}{n}
\rput{41.001434}(3.2350225,-6.07398){\rput(9.74,1.3040757){$\mcd_{\log\eps}^2(\Gamma)$}}
\psbezier[linewidth=0.04](11.36,1.5190758)(10.392135,1.2676057)(10.415929,1.746844)(9.82,0.75907576)(9.2240715,-0.22869256)(11.296868,-1.2518924)(12.16,-0.8209242)(13.023131,-0.38995606)(15.62,-0.9609242)(15.38,0.09907577)(15.14,1.1590757)(14.31877,1.3348954)(13.36,1.6190758)(12.401229,1.9032562)(12.327865,1.7705458)(11.36,1.5190758)
\psbezier[linewidth=0.04](2.82,1.4990758)(1.9171258,1.069171)(1.7963772,1.1587967)(1.82,0.15907577)(1.8436228,-0.8406452)(3.42,-1.4209242)(3.58,-1.9809242)(3.74,-2.5409243)(7.486769,-3.430828)(6.8,-2.0609243)(6.1132307,-0.6910203)(6.3,1.3790758)(5.46,1.2190758)(4.62,1.0590758)(3.7228742,1.9289806)(2.82,1.4990758)
\psbezier[linewidth=0.04](2.1695895,2.122485)(1.2769662,1.7047033)(1.1378895,1.5656667)(1.0285375,0.5514621)(0.9191854,-0.46274242)(0.33825517,-0.5373007)(0.83180434,-1.3974015)(1.3253535,-2.2575023)(4.2,-3.5809243)(5.7,-3.6809242)(7.2,-3.7809243)(9.4,-1.6609242)(7.8,-1.1809242)(6.2,-0.7009242)(6.98,1.7990757)(6.084579,1.8440758)(5.189158,1.8890758)(3.062213,2.5402665)(2.1695895,2.122485)
\end{pspicture} 
}
\caption{The domains $X_{j,\eps},$ $\mcd_{\log \eps}^j(\Gamma)$ and the sets $\Gamma_{j,\eps}.$}
\label{fig5}
\end{figure}

To prove that $(X_{j,\eps}, \Delta_{g_j})$ have the same boundary data at $\Gamma_\eps,$ we take Fourier transform of $\mcr_\pm$ in the variable $s.$ We know from
\eqref{equalradf} that  if $E_{j}\left(\frac{n}{2}+i\lambda,y',z\right)=\widehat{\mcr_{j,+}}(s,y,z),$ then
\begin{gather}
E_{1}\left(\frac{n}{2}+i\lambda,y',z\right)=E_{2}\left(\frac{n}{2}+i\lambda,y',z\right), \;\  y \in \Gamma, \;\ z=(x,y) \in [0,\eps) \times \Gamma. \label{equaleins}
\end{gather}
Using an argument due to Melrose, Lemma 3.2 of   \cite{megs},  S\'a Barreto showed that for any $\lambda\neq0$, the set of functions given by 
\begin{equation}\label{830}
v_{j}(z,\lambda)=\int_{\partial X_{j}}E_{j}\left(\frac{n}{2}+i\lambda,y',z\right) \phi(y')\dvol_{h_0}(y'),\quad j=1,2,\quad \phi\in C^{\infty}(\partial X_{j}),
\end{equation}
  is dense in the set of solutions of 
\begin{equation}\label{831}
\left(\Delta_{g_j}-\lambda^2-\frac{n^2}{4}\right)u=0\quad \text{in } X_{j,\epsilon},\quad j=1,2,
\end{equation}
in the Sobolev space $\displaystyle H^{k}(X_{j,\epsilon})$ for any $k\geq 2.$

  We know from \eqref{equaleins} that $\displaystyle E_{1}(n/2+i\lambda,y',x,y)=E_2(\novt+i\la,y',x,y)$ provided $x\in [0,\eps)$ and $y,y'\in \Gamma,$ and hence $v_{1}(x,y,\lambda)=v_{2}(x,y,\lambda),$ $(x,y)\in (0,\eps)\times \Gamma$, $\lambda\in\mathbb{R}\setminus\{0\}.$ Therefore their traces and normal derivatives at $\Gamma_{\epsilon}$ are equal at $\Gamma_\eps,$ and the density of this set implies that the same is true for solutions of \eqref{831}.

We now recall that the graph of the Calder\'on projector of $\displaystyle \Delta_{g_j}-\nsq-\la^2$ in $X_{j,\epsilon}$, $j=1,2,$ restricted to $\Gamma,$ which we denote by $C_{j,\la}$, is the closed subspace of $L^{2}(\Gamma)\times H^{1}(\Gamma)$ consisting of $(f,g)\in L^{2}(\Gamma)\times H^1(\Gamma)$ such that there exists $u$ satisfying
\begin{gather*}
\left(\Delta_{g_j}-\nsq-\la^2\right)u=0\quad \text{in } X_{j,\epsilon},\quad j=1,2, \\
u|_{\partial X_{j,\epsilon}}=f\in C_0^\infty(\Gamma),\quad \partial_{\nu}u|_{\Gamma}=g.
\end{gather*}
Therefore $\lambda^2$ is in the Dirichlet spectrum of $\displaystyle \Delta_{g_j}-\nsq$ in $X_{j,\epsilon}$ if and only if $C_{j,\lambda}$ contains a subspace of pairs of the form $(0,g),$ $g\neq 0.$ If $\la\in \mr,$ this only gives us eigenvalues greater than or equal to $0.$   But since $E_{j}(\novt+i\la,y,z)$ are meromorphically in $\im \la <0,$ they must also agree in this region, and therefore we obtain all eigenvalues.  Therefore we conclude that the spectral data 
$(\Delta_{g_j}, X_{j,\eps}),$ $ j=1,2$ coincide at $\Gamma.$ Thus the hypothesis of Proposition \eqref{pkl} are satisfied, and hence the maps
$\Psi_\eps$ satisfyng \eqref{familymaps} exist and  this concludes the proof of Theorem \ref{INV}.

\end{document}